\documentclass[11pt]{article}
\usepackage{amssymb, amsmath}
\usepackage{amsthm,array,hyperref,caption,subcaption, float,longtable,multirow}
\usepackage[margin=.8 in]{geometry}
\usepackage{enumerate}
\allowdisplaybreaks[1]
\numberwithin{equation}{section}

\newcolumntype{P}[1]{>{\centering\arraybackslash}p{#1}}
\newcolumntype{M}[1]{>{\centering\arraybackslash}m{#1}}

\newtheorem{example}{Example}[section]
\newtheorem{thm}{Theorem}[section]
\newtheorem{cor}{Corollary}[section]
\newtheorem{note}{Note}[section]
\newtheorem{notation}{Notation}[section]
\newtheorem{pro}{Proposition}[section]

\newtheorem{rem}{Remark}[section]

\newtheorem{defn}{Definition}[section]
\usepackage{graphicx}
\begin{document}
	\markboth{R. Rajkumar and M. Gayathri}{}
	\title{Spectra of generalized corona of graphs constrained by vertex subsets}
	
	\author {R. Rajkumar\footnote{e-mail: {\tt rrajmaths@yahoo.co.in}}~ and M. Gayathri\footnote{e-mail: {\tt mgayathri.maths@gmail.com}, }
		\\ \small \it Department of Mathematics,
		\small \it The Gandhigram Rural Institute-Deemed to be University,\\
		\small \it Gandhigram--624 302, Tamil Nadu, India.\\
		}
\date{}
	\maketitle
	
\begin{abstract}
	
		In this paper, we introduce a  generalization of corona of graphs. This construction  generalizes the generalized corona of graphs (consequently, the corona of graphs), the cluster of graphs, the corona-vertex subdivision graph of graphs and the corona-edge subdivision graph of graphs. Further, it enables to get some more variants of corona of graphs as its particular cases. To determine the spectra of the adjacency, Laplacian and the signless Laplacian matrices of the above mentioned graphs, we define a notion namely, the coronal of a matrix constrained by an index set, which generalizes the coronal of a graph matrix.
	Then we prove several results pertain to the determination of this value. Then we determine the characteristic polynomials of the adjacency and the Laplacian matrices of this graph in terms of the characteristic polynomials of the adjacency and the Laplacian matrices of the constituent graphs and the coronal of some matrices related to the constituent graphs.		 
	Using these, we derive the characteristic polynomials of the adjacency and the Laplacian matrices of the above mentioned existing variants of corona of graphs, and some more variants of corona of graphs with some special constraints.

\end{abstract}

\section{Introduction}
\subsection{Basic definitions and notations}
All the graphs assumed in this paper are undirected and simple.	
For a graph $G$ with $V(G) = \{v_1,v_2,\ldots, v_n\}$ and $E(G)=\{e_1,e_2,\ldots,e_m\}$, the \emph{adjacency matrix, vertex-edge incidence matrix} (or simply \emph{incidence matrix}), \emph{degree matrix, Laplacian matrix} and \emph{the signless Laplacian matrix of $G$} are denoted by $A(G)$, $B(G)$, $D(G)$, $L(G)$ and $Q(G)$, respectively, and are defined as follows: $A(G)=[a_{ij}]$,  where
$a_{ij}=1,$ if $i\neq j$ and, $v_i$ and $v_j$ are adjacent in $G$ for $i,j=1,2,\ldots,n$; 0, otherwise. $B(G)=[b_{ij}]$, where $b_{ij}=1,$ if the vertex $v_i$ is incident with the edge $e_j$ for $i=1,2,\ldots,n$ and $j=1,2,\ldots,m$; 0, otherwise. $D(G)=diag(d_1,d_2,\ldots,d_n)$, where $d_i$ denotes the degree of $v_i$ in $G$ for $i=1,2,\ldots,n$.  $L(G)=D(G)-A(G)$; $Q(G)=D(G)+A(G)$.
The characteristic polynomials of the adjacency, the Laplacian and the signless Laplacian matrices of $G$ are denoted by $P_{G}(x)$, $L_G(x)$ and $Q_G(x)$, respectively, and the eigenvalues of $A(G)$, $L(G)$ and $Q(G)$ are said to be the \emph{$A$-spectrum, the $L$-spectrum} and the \emph{$ Q$-spectrum of $G$}, respectively.
Two graphs are said to be \emph{$A$-cospectral} (resp. \emph{$L$-cospectral}, $ Q$-cospectral) if they have same $A$-spectrum (resp. $L$-spectrum, $Q$-spectrum).

Unless specifically mentioned otherwise, the $A$-spectrum and $L$-spectrum  of $G$ are denoted by 
$\lambda_1(G)\geq\lambda_2(G)\ldots\geq\lambda_n(G)$,
$0=\mu_1(G)\leq\mu_2(G)\leq\ldots\leq\mu_n(G)$,
respectively.


The complete graph on $n$ vertices is denoted by $K_n$ and the complete bipartite graph whose partite sets having $p$ and $q$ vertices is denoted by $K_{p,q}$. A semi-regular bipartite graph with parameters $(n_1,n_2,r_1,r_2)$ is a bipartite graph with bipartition $(X,Y)$ such that $|X|=n_1$, $|Y|=n_2$, the vertices in $X$ have degree $r_1$ and the vertices in $Y$ have degree $r_2$. The complement of a graph $G$ is denoted by $\overline G$. 	
Let $\mathcal R_{n\times m}(s)$ be the collection of all $n\times m$ real matrices $M$ such that the sum of the entries in each row of $M$ is equal to $s$. Let $\mathcal C_{n\times m}(c)$ be the collection of all real $n\times m$ matrices $M$ such that the sum of the entries in each column of $M$ is equal to $c$. Also, let $\mathcal{RC}_{n\times m}(s,c)$ be the collection of all $n\times m$ real matrices such that  $M\in \mathcal R_{n\times m}(s)$ and $M\in \mathcal C_{n\times m}(c)$.  Let $J_{n\times m}$ denotes the matrix of size $n\times m$ in which all the entries are 1, and $J_n$ denotes the matrix $J_{n\times n}$.

\subsection{Spectra of graphs constructed by graph operations}
The spectra of a graph reveal lots of information on the structural properties of that graph and the study of spectra of graphs has been found applications in variety of fields such as physics, chemistry, computer science, etc. (see \cite{bapat2010,brouwer2012,cvetkovic2010,cvetkovic2011}).


It is a common problem in spectral graph theory that to what extent the spectrum of a graph constructed using  graph operations can be described in terms of the spectrum of the constituting graph(s). Over the past five decades, considerable attention has been paid by the researchers on the spectra of graphs obtained using some graph operations such as union, Cartesian product, strong product, NEPS, rooted product, corona product, join, vertex deletion etc. For the results on the spectra of these graphs, we refer the reader to \cite{barik2018,cvetkovic2010,gayathri2019, rajkumar2019, rajkumar2020, sayama2016} and the references cited there in. 	

\subsubsection{Unary graph operations}
In the literature, several graph constructions have been made using one or more graphs. For the reader's convenience, here we recall the definitions of graphs constructed by some unary graph operations:  The \emph{subdivision graph $S(G)$ of $G$} is the graph obtained by inserting a new vertex into every edge of $G$.
The \emph{$R$-graph $R(G)$ of $G$} is the graph obtained by adding a new vertex for each edge of $G$, and joining the new vertex to the end vertices of the corresponding edge.
The \emph{$\mathcal Q$- graph $\mathcal Q(G)$ of $G$}  is the graph obtained from G by inserting a new vertex into each edge of $G$, and joining the new vertices which lie on adjacent edges of $G$.
The \emph{central graph $Ct(G)$ of $G$} is the graph obtained by taking one copy of $S(G)$ and joining the vertices which are not adjacent in $G$.
The \emph{total graph $T(G)$ of $G$}  is the graph whose vertices are the vertices together with the edges of $G$, and two vertices of $T(G)$ are adjacent if and only if the corresponding elements of $G$ are either adjacent or incident.
The \emph{quasi-total graph $\mathcal QT(G)$ of $G$} is the graph obtained by taking one copy of $\mathcal Q(G)$ and joining the vertices in $G$ which are not adjacent in $G$.
The \emph{duplication graph $Du(G)$ of $G$}  is a graph obtained by taking new vertices corresponding to each vertex of $G$ and joining the new vertex to the vertices in $G$ which are adjacent to the corresponding vertex in $G$ of the new vertex and deleting the edges of $G$.  The \emph{$C$-graph $C(G)$ of $G$} \cite{adiga2015}  is the graph obtained by taking one copy of $G$ and $|V(G)|$ number of new vertices, and joining the $i$-th new vertex to the $i$-th vertex of $G$. The \emph{$N$-graph $N(G)$ of $G$} \cite{adiga2015}  is the graph obtained by taking one copy of $G$ and $|V(G)|$ number of new vertices, and joining the $i$-th new vertex to the vertices which are adjacent to the $i$-th vertex of $G$. 

Further, the following unary graph operations are defined the spectra of the graphs obtained by them are studied in \cite{rajkumar2019}: The \emph{point complete subdivision graph of $G$} is the graph obtained by taking one copy of $S(G)$ and joining all the vertices $v_i,v_j\in V(G)$. 	The \emph{$\mathcal Q$-complemented  graph of $G$} is the graph obtained by taking one copy of $S(G)$ and joining the new vertices which lie on the non-adjacent edges of $G$.
The \emph{total complemented graph of $G$} is the graph obtained by taking one copy of $R(G)$ and joining the new vertices lie which on the non-adjacent edges of $G$. The \emph{quasitotal complemented graph of $G$} is the graph obtained by taking one copy of $\mathcal Q$-complemented graph of $G$ and joining all the vertices $v_i,v_j\in V(G)$ which are not adjacent in $G$.
The \emph{complete $\mathcal Q$-complemented graph of $G$} is the graph obtained by taking one copy of $\mathcal Q$-complemented graph of $G$ and joining all the vertices of $v_i,v_j\in V(G)$.
The \emph{complete subdivision graph of $G$} is the graph obtained by taking one copy of $S(G)$ and joining the all the new vertices which lie on the edges of $G$.
The  \emph{complete $R$-graph of $G$} is the graph obtained by taking one copy of $R(G)$ and joining all the new  vertices which lie on the edges of $G$.	The  \emph{complete central graph of $G$} is the graph obtained by taking one copy of central graph of $G$ and joining all the new vertices which lie on the edges of $G$.
The  \emph{fully complete subdivision graph of $G$} is the graph obtained by taking one copy of $S(G)$ and joining all the vertices of $G$ and joining all the new vertices which lie on the edges of $G$. 

Let $\mathcal U$ be the set of all unary graph operations mentioned above.
The set of new vertices in $U(G)$ for a graph $G$ and $U\in\mathcal U$ is commonly denoted by $I(G)$.

\subsubsection{Corona of graphs and some of its variants}\label{sec corona product of graphs}

The corona of graphs is one of the well-known graph operation which has been attracted the attention of many researchers. In 1970, the corona of two graphs was first introduced by Frucht and Harary to construct a graph whose automorphism group is the wreath product of the automorphism group of their components~\cite{frucht1970}. 
Following this, several variants of corona of graphs such as the edge corona \cite{hou2010}, the neighbourhood corona \cite{indulal2011}, the subdivision vertex corona, the subdivision  edge corona \cite{lu2013}, the subdivision vertex neighbourhood corona, the subdivision  edge neighbourhood corona \cite{liu2013}, the subdivision double corona and the subdivision double neighbourhood corona \cite{barik2016} have been defined and their spectral properties were studied. 

Below we give the definitions of corona of graphs and some of its variants which are used in this paper:
Let $G$ be a graph with $n$ vertices and $m$ edges, and let $H$ be a graph.
The \emph{corona of $G$ and  $H$}  is the graph obtained by taking one copy of $G$ and $n$ copies of $H$, and joining the  $i$-th vertex of $G$ to all the vertices of $i$-th copy of $H$ for $i=1,2,\ldots,n$.
In the same paper, the following variant of corona of graphs was defined.
\normalfont
The \emph{cluster of $G$ and a rooted graph $H$}, denoted by $G\{H\}$, is the graph obtained by taking one copy of $G$ and $n$ copies of $H$, and joining the $i$-th vertex of $G$ to the root vertex of the $i$-th copy of $H$ for $i=1,2,\ldots,n$.
Barik et al.~\cite{barik2007} studied the spectral properties of corona of graphs. They have obtained the $A$-spectrum (resp. $L$-spectrum) of the corona of $G$ and $H$ for any graph $G$ and a regular graph $H$ (resp. for any graph $G$ and $H$), in terms of the $A$-spectrum (resp. $L$-spectrum) of $G$ and $H$ by determining its eigenvectors. McLeman and McNicholas \cite{mcleman2011} computed the $A$-spectrum of the corona of any pair of graphs using a new graph invariant called the coronal of a graph. Cui and Tian \cite{cui2012} determined the characteristic polynomial of the signless Laplacian matrix of corona of two arbitrary graphs by using the coronal of a graph matrix.  Wang and Zhou \cite{wang20131} obtained the signless Laplacian spectrum of corona of $G$ and $H$, when $H$ is regular, by determining its eigenvectors. Liu \cite{qliu2014} obtained the characteristic polynomial of the Laplacian matrix of the corona of graphs.
Lu and Miao \cite{lu2014} introduced the following two variants of corona of graphs:
The \emph{corona-vertex subdivision graph of $G$ and $H$} is the graph obtained by taking one copy of $G$ and $n$ copies of $S(H)$, and joining the $i$-th vertex of $G$ to all the vertices of the $i$-th copy of $V(H)$ for $i=1,2,\ldots,n$.
The \emph{corona-edge subdivision graph of $G$ and $H$} is the graph obtained by taking one copy of $G$ and $n$ copies of $S(H)$, and joining the $i$-th vertex of $G$ to all the vertices of the $i$-th copy of $I(H)$ for $i=1,2,\ldots,n$.
Laali et al.\cite{laali2016} defined the \emph{generalized corona of graphs}, in which they replaced the $n$ copies of $H$ by the graphs $H_1,H_2,\ldots,H_n$ in the definition of corona of $G$ and $H$, and obtained its characteristic polynomials of the adjacency, the Laplacian and the signless Laplacian matrices.

\subsection{Scope of the paper}
Motivated by the above, we define the following.
\begin{defn}\label{defn gen corona constrained by vertex subset}
	\normalfont
	Let $G$ be  a graph with $V(G)=\{v_1,v_2,\ldots, v_n\}$. Let  $\mathcal{H}$ be a sequence of $n$ graphs $H_1,H_2,\ldots, H_n$ and $\mathcal T$ be a sequence of sets $T_1,T_2,\ldots,T_n$, where $T_i\subseteq V(H_i)$, $i=1,2,\ldots, n$. Then the \emph{generalized corona of $G$ and $\mathcal H$ constrained by $\mathcal{T}$}, denoted by $G\circledast_{\mathcal T}\mathcal{H}$, is the graph obtained by taking one copy of $G$, $H_1,H_2, \ldots, H_n$, and joining the vertex $v_i$ to all the vertices in $T_i$ for $i=1,2,\ldots, n$.
\end{defn}

The above definition introduces a new way of generalization in corona of graphs, in which the
base graphs are joined to the vertices in a vertex subset of the constituent graphs
instead of joining all the vertices. Further, it generalizes the cluster of two graphs and some of the variants of corona of graphs:
Taking $H_i=H$ and $T_i=\{\text{the root vertex of } H\}$ for $i=1,2,\ldots,n$ in the preceding definition, we get the cluster of $G$ and $H$; Taking $T_i=V(H_i)$ for $i=1,2,\ldots,n$ in the preceding definition, we get the generalized corona of $G$ and $H_1$, $H_2$,$\ldots,$ $H_n$. We denote this graph simply by $G\circledast\mathcal H$; Taking $H_i=S(H)$ and $T_i=V(H)$ for each $i=1,2,\ldots,n$, we get the corona-vertex subdivision graph $G$ and $H$; Taking $H_i=S(H)$ and $T_i=I(H)$ for each $i=1,2,\ldots,n$, we get the corona-edge subdivision graph $G$ and $H$.

Moreover, for each $U\in\mathcal U$, if we take $H_i= U(H_i')$ for a graph $H_i'$ and $T_i=V(H_i)$ or $I(H_i)$ for each $i=1,2,\ldots,n$ in Definition~\ref{defn gen corona constrained by vertex subset}, we get some more new variants of corona of graphs. Notice that if for each $i=1,2,\ldots,n$, $H_i=Du(H_i')$, $T_i=V(H_i')$ and $T_i'=I(H_i')$, then the graphs $G\circledast_{\mathcal T}\mathcal H$ and $G\circledast_{\mathcal T'}\mathcal H'$ are isomorphic, where $\mathcal H$ is the sequence of graphs $H_1,H_2,\ldots, H_n$, and  $\mathcal T$ (resp. $\mathcal T'$)  is the sequence $T_1,T_2,\ldots,T_n$ (resp. $T_1',T_2',\ldots,T_n'$).
\begin{example}\label{examplegencorona}
	\normalfont
	The graphs $G,H_1,H_2,H_3,H_4,H_5$ and $G\circledast_{\mathcal T}\mathcal H$ are shown in Figure \ref{fgencorona}, where $\mathcal{H}$ is the sequence $H_1,H_2,H_3,H_4,H_5$, and $\mathcal{T}$ $T_1,T_2,T_3,T_4,T_5$ is the sequence with $T_1=\{u_3\}$, $T_2=\{x_1,x_3\}$, $T_3=\{w_1,w_3\}$, $T_4=\{t_2\}$ and $T_5=\{s_1,s_2\}$.
	To ease the identification of vertices, we colored the vertices in $T_i, i= 1,2,\ldots,5$ with yellow. For each $i=1,2,\ldots,5$, the $i$-th vertex of $G$ and the edges of $H_i$ are colored with the same color.	
\end{example}

\begin{figure}[ht!]
	\begin{center}
		\includegraphics[scale=.7]{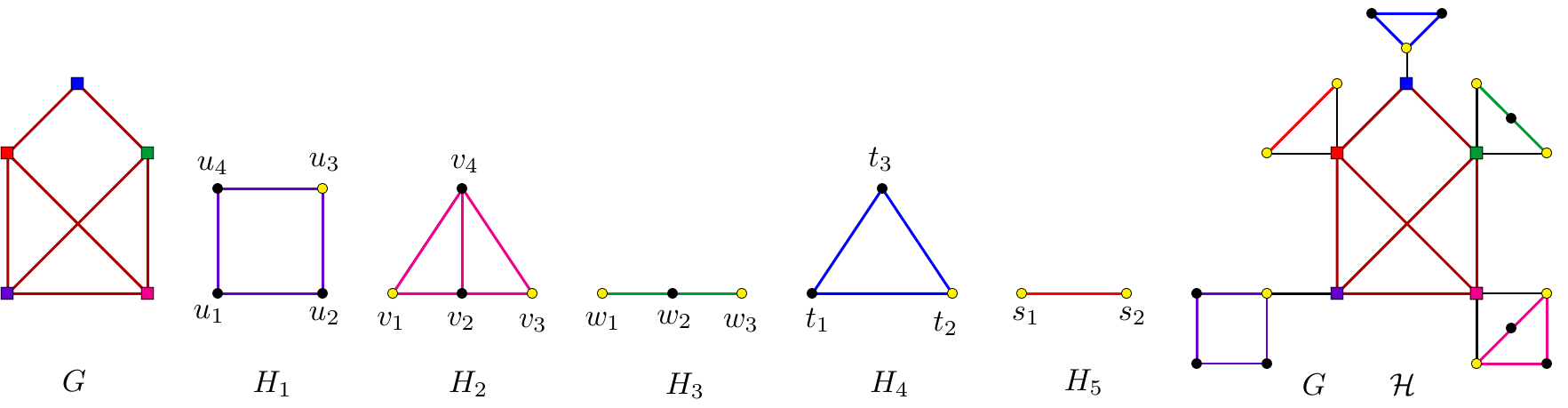}
		\put(-55,2){$\circledast_{\mathcal T}$}
		\caption{The generalized corona of $G$ and $\mathcal H$ constrained by $\mathcal T$}
		\label{fgencorona}
	\end{center}
\end{figure}

The rest of the paper is arranged as follows: In Section~2, we define the coronal of a matrix constrained by an index set and the coronal of a graph constrained by a vertex subset. We determine the coronal of some special kind of matrices. Also, we obtain the coronal of a matrix constrained by an arbitrary index set in terms of the coronal of some other matrix related to the given matrix.
Using these, we determine the coronal of the graphs constrained by some of their vertex subsets obtained by the unary graph operations in $\mathcal U$, when the base graph is regular, the coronal of a semi-regular bipartite graph,  the complete graph, complete bipartite graphs. In Section~3 and 4, we determine the characteristic polynomials of the adjacency and Laplacian matrices of the generalized corona of graphs constrained by vertex subsets, respectively. Further, we deduce the characteristic polynomials of the adjacency and the Laplacian matrices of some existing corona of graphs and the new variants of corona of graphs.

\section{Coronal of a matrix constrained by an index set}\label{sec Coronal of a matrix constrained by an index set}


McLeman et al. introduced the notion of coronal of a graph:

\begin{defn}(\cite{mcleman2011})
	\normalfont
	Let $H$ be a graph with $n$ vertices. Then the sum of
	the entries of the matrix $(xI_n - A(H))^{-1}$ is said to be the \emph{coronal $\Gamma_H(x)$ of $H$}. This can be calculated as
	$$\Gamma_H(x) = J_{1\times n} (xI_n - A(H))^{-1}J_{n\times1}.$$
\end{defn}

Cui and Tian generalized this concept as follows:

\begin{defn}(\cite{cui2012})\label{defn matrix coronal}
	\normalfont
	Let $G$ be a graph of with $n$ vertices and $M$ be a graph matrix of $G$.
	Then the sum of the entries of the matrix $(xI_n-M)^{-1}$ is said to be the \emph{$M$-coronal of $G$}, and is denoted by $\Gamma_M(x)$. That
	is $$\Gamma_M(x)=J_{1\times n}(xI_n-M)^{-1}J_{n\times 1}.$$
\end{defn}


For a subset $B$ of a set $A=\{u_1,u_2,\ldots,u_n\}$, the \emph{indicator vector of $B$ (with respect to $A$)} is a $0-1$ vector of length $n$ in which the $i$-th coordinate is 1 or 0, according as $u_i\in B$ or $u_i\notin B$, and it is denoted by $\mathbf r_B$.
For a matrix $M\in M_n(\mathbb R)$ and an index set $\alpha\subseteq \{1,2\ldots,n\}$, \textit{the principal submatrix of $M$} formed by $\alpha$ is the (sub)matrix of entries that lie in the rows and columns indexed by $\alpha$.

In the following definition, we introduce the notion of coronal of a matrix constrained by an index set, which generalizes Definition~\ref{defn matrix coronal}.

\begin{defn}
	\normalfont
	Let $M\in M_n(\mathbb R)$ and $\alpha\subseteq \{1,2,\ldots,n\}$ be an index set. Then  \emph{the coronal of $M$ constrained by $\alpha$}, denoted by $\Gamma_M^{\alpha}(x)$, is defined as the sum of all entries in the principal submatrix of $(xI_n-M)^{-1}$ formed by $\alpha$. Notice that this can be calculated by $$\Gamma_M^{\alpha}(x)=\mathbf{r}_{\alpha}(xI_n-M)^{-1}\mathbf{r}_{\alpha}^T.$$ 
\end{defn}

In the above definition, $xI_n-M$ is viewed as a matrix over the field of rational functions $\mathbb{C}(x)$. So $xI_n-M$ is invertible.

\begin{rem}
	\normalfont
	\begin{enumerate}[(1)]
		\item If $\alpha=\{1,2,\ldots,n\}$, then we denote $\Gamma_M^{\alpha}(x)$ simply by $\Gamma_M(x)$ and we call this simply as the \textit{coronal of $M$}. Notice that $\Gamma_M(x)=J_{n\times 1}(xI_n-M)^{-1}J_{1\times n}$.
		
		\item If $H$ is a graph, $T\subseteq V(H)$ and $M$ is a graph matrix of $H$, then we call $\Gamma_M^{T}(x)$ as the \textit{$M$-coronal of $H$ constrained by the vertex subset $T$}. 	If $T=V(H)$, then $\Gamma_M^{T}(x)=\Gamma_M(x)$.
		For $M=A(H)$ (resp. $L(H)$, $Q(H)$), we call $\Gamma_M^T(x)$ as the coronal (resp. $L$-coronal, $Q$-coronal) of $H$ constrained by the vertex subset $T$.
		
		\item If $\alpha=\{i\}$, then we denote $\Gamma_M^{\alpha}(x)$ simply by $\Gamma_M^i(x)$. Notice that $\Gamma_M^i(x)$ is the $i$-th diagonal entry of the matrix $(xI_n-M)^{-1}$.
	\end{enumerate}
\end{rem}

%
%
%
%
%

The following result gives the coronal of a matrix $M\in\mathcal R_{n\times n}(s)$ for some $s\in \mathbb R$.
\begin{pro}(\cite[pro 2]{cui2012})\label{equal row sum matrices}
	If $M\in\mathcal R_{n\times n}(s)$, then $\Gamma_M(x)=\displaystyle\frac{n}{x-s}$.
\end{pro}

In the next result, we show that the coronal of a matrix is invariant under the rearrangement of the same rows and columns of the matrix.
\begin{pro}\label{coronal of a matrix and a rearranging matrix}
	If $A$ and $B$ are square matrices of order $n$ such that $PAP^T=B$ for a permutation matrix $P$, then  $\Gamma_A(x)=\Gamma_B(x)$.
\end{pro}

\begin{proof}
	\begin{eqnarray}
		\Gamma_A(x)&=&J_{1\times n}(xI_n-A)^{-1}J_{n\times 1}\notag
		\\&=&
		J_{1\times n}(xI_n-P^TBP)^{-1}J_{n\times 1}\notag
		\\&=&
		J_{1\times n}P^T(xI_n-B)^{-1}PJ_{n\times 1}\notag
		\\&=&
		J_{1\times n}(xI_n-B)^{-1}J_{n\times 1}\notag
		\\&=&
		\Gamma_{B}(x).\notag
	\end{eqnarray}
\end{proof}

In the following result, we obtain the coronal of a matrix, which satisfies some special constraints.
\begin{thm}\label{coronal and corornal constrained by index of partitioned matrix}
	Let $A$ be square matrix of order $n$ such that
	\[A=\begin{bmatrix}
	A_1&A_2\\A_3&A_4
	\end{bmatrix},
	\] where $A_1\in \mathcal R_{n_1\times n_1}(a_1)$, $A_2\in \mathcal R_{n_1\times n_2}(a_2)$, $A_3\in \mathcal R_{n_2\times n_1}(a_3)$ and $A_4\in \mathcal R_{n_2\times n_2}(a_4)$. Then
	\begin{eqnarray}
		\Gamma_{A}(x)=\displaystyle\frac{(n_1+n_2)x+n_1(a_2-a_4)+n_2(a_3-a_1)}{x^2-(a_1+a_4)x+a_1a_4-a_2a_3}. \label{equation coronalvalues of partitioned matrix}
	\end{eqnarray}
\end{thm}

\begin{proof}
	It can be verified that
	\begin{eqnarray}
		\Gamma_{A}(x)&=& [J_{1\times n_1}~J_{1\times n_2}] (xI_{n_1+n_2}-A)^{-1}[J_{1\times n_1}~J_{1\times n_2}]^T\label{equation6}.
	\end{eqnarray}
	
	By using \cite[(0.7.3.1)]{horn1985}, we have
	\begin{equation}
		(xI_{n_1+n_2}-A)^{-1}= 	\begin{bmatrix}
			A_1' &-A_2'
			\\-A_3'&A_4'
		\end{bmatrix}, \label{inverse of xI-A partitioned}
	\end{equation}
	where 
	
	$A_1'=\left( xI_{n_1}-A_1-A_2\left[xI_{n_2}-A_4\right]^{-1}A_3 \right)^{-1}$,
	
	$A_2'=(xI_{n_1}-A_1)^{-1}A_2\left(A_3[xI_{n_1}-A_1]^{-1}A_2-\left[xI_{n_2}-A_4 \right] \right)^{-1}$,
	
	$A_3'=\left(A_3[xI_{n_1}-A_1]^{-1}A_2-\left[xI_{n_2}-A_4 \right] \right)^{-1}A_3(xI_{n_1}-A_1)^{-1}$,

	$A_4'=\left(xI_{n_2}-A_4-A_3[xI_{n_1}-A_1]^{-1}A_2\right)^{-1}$.

	\normalsize
	So, \eqref{equation6} becomes,
	\begin{eqnarray}
		\Gamma_{A}(x)&=&  S_1-S_2-S_3+S_4\label{equationmain},
	\end{eqnarray}
	where
	$S_1=J_{1\times n_1}A_1' J_{n_1\times 1}$,
	$S_2=J_{1\times n_1}A_2'J_{n_2\times 1}$,
	$S_3=J_{1\times n_2}A_3'J_{n_1\times 1}$,
	$S_4=
	J_{1\times n_2}A_4' J_{n_2\times 1}.$
	
	Since  $A_1\in \mathcal R_{n_1\times n_1}(a_1)$, $A_2\in \mathcal R_{n_1\times n_2}(a_2)$, $A_3\in \mathcal R_{n_2\times n_1}(a_3)$ and $A_4\in \mathcal R_{n_2\times n_2}(a_4)$, we have
	\begin{eqnarray}
		A_1 J_{n_1\times 1} = a_1J_{n_1\times 1}\label{equation row sum of A1}
		\\
		A_2 J_{n_2\times 1} = a_2J_{n_1\times 1}\label{equation row sum of A2}
		\\
		A_3 J_{n_1\times 1} = a_3J_{n_2\times 1}\label{equation row sum of A3}
		\\
		A_4 J_{n_2\times 1} = a_4J_{n_2\times 1}\label{equation row sum of A4}.
	\end{eqnarray}
	
	Also notice that, the sum of the entries in each row of $(xI_{n_1}-A_1)^{-1}$ is equal to $\displaystyle\frac{1}{x-a_1}$.
	So,
	\begin{eqnarray}
		(xI_{n_1}-A_1)^{-1}J_{n_1\times 1}=\left( \displaystyle\frac{1}{x-a_1}\right) J_{n_1\times 1}\label{equation rowsum of A1inverse}.
	\end{eqnarray}
	
	Similarly,
	\begin{eqnarray}
		(xI_{n_2}-A_4)^{-1}J_{n_2\times 1}=\left( \displaystyle\frac{1}{x-a_4}\right) J_{n_2\times 1}\label{equation rowsum of A4inverse}.
	\end{eqnarray}
	
	By using \eqref{equation row sum of A3}, \eqref{equation rowsum of A4inverse} and \eqref{equation row sum of A2}, we have
	\begin{eqnarray}
		\left( A_2(xI_{n_2}-A_4)^{-1}A_3 \right) J_{n_1\times 1}&=&
		A_2\left( xI_{n_2}-A_4 \right)^{-1} \left( A_3J_{n_1\times 1}\right)\notag
		\\&=&
		a_3A_2 (xI_{n_2}-A_4)^{-1}  J_{n_2\times 1}\notag
		\\&=&
		\left( \displaystyle\frac{a_3}{x-a_4}\right) A_2J_{n_2\times 1}\notag
		\\&=&
		\left( \displaystyle\frac{a_2a_3}{x-a_4}\right) J_{n_1\times 1}.\label{equation7}
	\end{eqnarray}
	
	Similarly, we get
	\begin{eqnarray}\label{a3a1a2jn2}
		\left( A_3(xI_{n_1}-A_1)^{-1}A_2 \right) J_{n_2\times 1}&=&
		\left( \displaystyle\frac{a_2a_3}{x-a_1}\right) J_{n_2\times 1}.
	\end{eqnarray}
	
	So by \eqref{equation row sum of A1} and \eqref{equation7}, we have
	\begin{eqnarray}
		\left[ xI_{n_1}-A_1-A_2(xI_{n_2}-A_4)^{-1}A_3 \right]J_{n_1\times 1}
		&=&
		\left(x-a_1-\displaystyle\frac{a_2a_3}{x-a_4}\right) J_{n_1\times 1}\notag
		\\&=&\left( \displaystyle\frac{x^2-(a_1+a_4)x+a_1a_4-a_2a_3}{x-a_4}\right) J_{n_1\times 1}.	\notag
	\end{eqnarray}
	
	Consequently,	\begin{eqnarray}
		\left[ xI_{n_1}-A_1-A_2(xI_{n_2}-A_4)^{-1}A_3 \right]^{-1}J_{n_1\times 1}
		&=&\left( \displaystyle\frac{x-a_4}{x^2-(a_1+a_4)x+a_1a_4-a_2a_3}\right) J_{n_1\times 1}.	\notag
	\end{eqnarray}
	
	So, we have,
	\begin{eqnarray}
		S_1&=&\displaystyle\frac{n_1(x-a_4)}{x^2-(a_1+a_4)x+a_1a_4-a_2a_3}.\label{equation values of S1}
	\end{eqnarray}
	
	Similarly, we get
	\begin{eqnarray}
		S_4&=&\displaystyle\frac{n_2(x-a_1)}{x^2-(a_1+a_4)x+a_1a_4-a_2a_3}.\notag
	\end{eqnarray}
	
	Now, by using \eqref{a3a1a2jn2} and \eqref{equation rowsum of A4inverse}, we have \begin{eqnarray}
		\left\{A_3[I_{n_1}-A_1]^{-1}A_2-\left(xI_{n_2}-A_4 \right) \right\}^{-1}J_{n_2\times 1}&=&\displaystyle\frac{1}{\displaystyle\left( \frac{a_2a_3}{x-a_1}\right) -(x-a_4)}J_{n_2\times 1}\notag
		\\&=&\displaystyle\frac{x-a_1}{x^2-(a_1+a_4)x+a_1a_4-a_2a_3}J_{n_2\times 1}\notag\\\label{eqn9}.
	\end{eqnarray}
	
	Using \eqref{equation rowsum of A1inverse} and \eqref{eqn9}, we get
	\begin{eqnarray}
		S_2
		&=& J_{1\times n_1}(xI_{n_1}-A_1)^{-1}A_2\left\{A_3(I_{n_1}-A_1)^{-1}A_2-\left(xI_{n_2}-A_4 \right) \right\}^{-1}J_{n_2\times 1}\notag
		\\&=&\left( \frac{1}{x-a_1}\right) J_{1\times n_1}A_2\left( \frac{x-a_1}{x^2-(a_1+a_4)x+a_1a_4-a_2a_3}\right) J_{n_1\times 1}\notag \\&=& \left( \frac{a_2}{x^2-(a_1+a_4)x+a_1a_4-a_2a_3}\right)J_{1\times n_1} J_{n_1\times 1}\notag
		\\&=& \frac{n_1a_2}{x^2-(a_1+a_4)x+a_1a_4-a_2a_3}.\notag
	\end{eqnarray}
	
	Also, we get
	\begin{eqnarray}
		S_3&=&	J_{1\times n_2}\left\{A_3(xI_{n_1}-A_1)^{-1}A_2-\left(xI_{n_2}-A_4 \right) \right\}^{-1}A_3(xI_{n_1}-A_1)^{-1}J_{n_1\times 1}\notag
		\\&=& J_{1\times n_1}\left( \frac{x-a_1}{x^2-(a_1+a_4)x+a_1a_4-a_2a_3}\right)A_3\left( \frac{1}{x-a_1}\right) J_{n_2\times 1}\notag
		\\&=& \left( \frac{a_3}{x^2-(a_1+a_4)x+a_1a_4-a_2a_3}\right)J_{1\times n_2} J_{n_2\times 1}\notag
		\\&=& \frac{n_2a_3}{x^2-(a_1+a_4)x+a_1a_4-a_2a_3}\notag.
	\end{eqnarray}
	
	Substituting the values of $S_1,S_2,S_3$ and $S_4$ in \eqref{equationmain}, we get the result.	
\end{proof}
In view of Proposition~\ref{coronal of a matrix and a rearranging matrix}, if a matrix $A'$ can be transformed (by rearranging the same rows and columns of $A'$) to the matrix $A$ of the form given in Theorem~\ref{coronal and corornal constrained by index of partitioned matrix}, then the coronal of $A'$ can be determined by \eqref{equation coronalvalues of partitioned matrix constrained by index two}.

In the next result, we determine the coronal of a matrix constrained by an arbitrary index set in terms of the coronal of a matrix related to the given matrix. Also we prove that, the coronal of a matrix constrained by an arbitrary index set with $n_1$ elements is same as the coronal of a matrix obtained by a suitable rearrangement of the rows and columns of the given matrix constrained by the index set $\{1,2,\ldots,n_1\}$.
\begin{thm}\label{coronalvalues of partitioned matrix constrained by index two}
	Let $A$ be a square matrix of order $n$ and let $\alpha=\{\alpha_1,\alpha_2,\ldots,\alpha_{n_1}\}\subseteq \{1,2,\ldots,n\}$. Consider the partitioned matrix
	\[A'=\begin{bmatrix}
	A_1&A_2\\A_3&A_4
	\end{bmatrix},
	\]
	where $A_1$ is the principal submatrix of $A$ formed by $\alpha$, $A_2$ is the submatrix of $A$ formed by the rows in $\alpha$ and the columns in $\alpha^c$, $A_3$ is the submatrix of $A$ formed by the rows in $\alpha^c$ and the columns in $\alpha$, and $A_4$ is the principal submatrix of $A$ formed by $\alpha^c$. Then $$\Gamma_A^{\alpha}(x)=\Gamma_{A'}^{\alpha'}(x)=\Gamma_M(x),$$	
	where  $M=A_1+A_2(xI_{n_2}-A_4)^{-1}A_3 \text{ and } \alpha'=\{1,2,\ldots,n_1\}.$
	
	Moreover, if $A_1\in \mathcal R_{n_1\times n_1}(a_1)$, $A_2\in \mathcal R_{n_1\times n_2}(a_2)$, $A_3\in \mathcal R_{n_2\times n_1}(a_3)$ and $A_4\in \mathcal R_{n_2\times n_2}(a_4)$, then
	\begin{eqnarray}
		\Gamma_A^{\alpha}(x)=\displaystyle\frac{n_1(x-a_4)}{x^2-(a_1+a_4)x+a_1a_4-a_2a_3}.\label{equation coronalvalues of partitioned matrix constrained by index two}
	\end{eqnarray}
\end{thm}

\begin{proof}
	First we prove that $\Gamma_{A}^{\alpha}(x)=\Gamma_{A'}^{\alpha'}(x)$. Without loss of generality we assume that $\alpha_1 < \alpha_2 < \cdots < \alpha_{n_1}$.
	Let $p$ be a permutation on $\{1,2,\ldots,n\}$ such that $p(1)=\alpha_1, p(2)=\alpha_2,\ldots, p(n_1)=\alpha_{n_1}$ and $P$ be the permutation matrix corresponding to $p$.	
	%
	Then we have, $A'=PAP^T$. Notice that $[J_{1\times n_1}~\textbf{0}]$ is the indicator vector of $\alpha'$.	
	Now,	
	\begin{eqnarray}
		\Gamma_{A}^{\alpha}(x)&=&\mathbf{r}_{\alpha}(xI_n-A)^{-1}\mathbf{r}_{\alpha}^T\notag
		\\&=&
		\mathbf{r}_{\alpha}(xI_n-P^TA'P)^{-1}\mathbf{r}_{\alpha}^T\notag
		\\&=&
		\mathbf{r}_{\alpha}P^T(xI_n-A')^{-1}P\mathbf{r}_{\alpha}^T\notag
		\\&=&
		[J_{1\times n_1}~\textbf{0}] (xI_n-A')^{-1}[J_{1\times n_1}~\textbf{0}]^T\notag
		\\&=&
		\Gamma_{A'}^{\alpha'}(x)\notag.
	\end{eqnarray}
	
	Using \eqref{inverse of xI-A partitioned}, we have
	\begin{eqnarray}
		\Gamma_{A'}^{\alpha'}(x)&=&[J_{1\times n_1}~ \textbf{0}](xI_n-A')^{-1}[J_{1\times n_1}~\textbf{0}]^T\notag
		\\&=&J_{1\times n_1}\left[ xI_{n_1}-A_1-A_2(xI_{n_2}-A_4)^{-1}A_3 \right]^{-1}J_{n_1\times 1}
		\label{equation Gamma partition matrix2}
		\\&=&
		J_{1\times n_1}\left[ xI_{n_1}-M \right]^{-1}J_{n_1\times 1}\notag
		\\	&=&\Gamma_M(x). \notag
	\end{eqnarray}
	
	Consequently, we have $\Gamma_A^{\alpha}(x)=\Gamma_{A'}^{\alpha'}(x)=\Gamma_M(x).$
	
	If $A_1\in \mathcal R_{n_1\times n_1}(a_1)$, $A_2\in \mathcal R_{n_1\times n_2}(a_2)$, $A_3\in \mathcal R_{n_2\times n_1}(a_3)$ and $A_4\in \mathcal R_{n_2\times n_2}(a_4)$, then  by substituting the value of $S_1$ as given in \eqref{equation values of S1} in \eqref{equation Gamma partition matrix2}, we get \eqref{equation coronalvalues of partitioned matrix constrained by index two}.
	This completes the proof.
\end{proof}

Next,  we  start to determine the coronals of some classes of graphs constrained by some of their vertex subsets by using the previous results.
It is well-known that the adjacency matrices of the graph $U(G)$ for a graph $G$ and $U\in\mathcal U$ is of the form \[\begin{bmatrix}
A_1& A_2\\ A_3 & A_4
\end{bmatrix},\] where $A_1$, $A_2$, $A_4$ are as mentioned in the first row against each of these graphs in Table~\ref{table coronals of various graphs} and $A_3=A_2^T$.

\begin{cor}
	Let $G$ be an $r$-regular graph with $n$ vertices. Then the coronals of the graph $U(G)$, where $U\in\mathcal U$ constrained by some of their vertex subsets $T$ can be obtained  by using Table \ref{table coronals of various graphs}: For the vertex subsets in first row given against each these graphs, apply the values $a_1$, $a_2$, $a_3$ and $a_4$ in  \eqref{equation coronalvalues of partitioned matrix} and for the vertex subsets $V(G)$ and $I(G)$ in second and third rows given against each of these graphs, apply the values $a_1$, $a_2$, $a_3$ and $a_4$ in \eqref{equation coronalvalues of partitioned matrix constrained by index two}. Notice that in each of these cases, $A_3=A_2^T$.
\end{cor}

\begin{footnotesize}
	\begin{longtable}[h!]{ |m{.7cm} | M{1.6cm}|m{0.8 cm} | m{1.2cm} | m{.8cm} |m{1.2cm} | m{1.5cm} | m{1cm} | m{1cm} |  m{1.5cm} m{0cm}| }
		\hline
		\textbf{S. No} & \textbf{Graph ($G'$)} &  \textbf{Vertex subset}\boldmath {~$T$}& {\boldmath$A_1$} & {\boldmath$A_2$} & {\boldmath$A_4$}& {\boldmath$a_1$} &{\boldmath $a_2$} & {\boldmath$a_3$} & {\boldmath$a_4$}&\\
		\hline
		\multirow{3}{*}{1.} & \multirow{3}{*}{\parbox{1.7cm}{Subdivision graph of $G$}}  & $V(G')$
		&\textbf{0} & $B(G)$ & \textbf{0} & 0 & $r$ & 2  &0&\\[8pt]
		\cline{3-11}
		
		& & $V(G)$
		&\textbf{0} & $B(G)$ & \textbf{0} & 0 & $r$ & 2  &0&\\[8pt]
		\cline{3-11}
		& & $I(G)$ &\textbf{0 }& $B(G)^T$ & \textbf{0} & 0 & 2 & $r$  &0&\\[8pt]
		\hline
		
		\multirow{3}{*}{2.} & 	\multirow{3}{*}{\parbox{1.7cm}{$R$-graph of $G$}}  & $V(G')$ & $A(G)$ & $B(G)$ &  \textbf{0 }& $r$ & $r$ & 2 &0& \\[8pt]
		\cline{3-11}
		
		&  & $V(G)$ & $A(G)$ & $B(G)$ & \textbf{0 } & $r$ & $r$ & 2 &0& \\[8pt]
		\cline{3-11}
		&   & $I(G)$ & \textbf{0 } & $B(G)^T$ & $A(G)$ & 0 & 2 & $r$ &$r$ &\\[8pt]
		\hline
		
		\multirow{3}{*}{3.} &  \multirow{3}{*}{\parbox{1.7cm}{$\mathcal Q$-graph of $G$}}  & $V(G')$ & \textbf{0 } & $B(G)$ & $A(\mathcal L(G))$ & 0  & $r$ & 2  &$2r-2$&\\[8pt]
		\cline{3-11}
		
		& & $V(G)$ &\textbf{0 } & $B(G)$ & $A(\mathcal L(G))$ & 0  & $r$ & 2  &$2r-2$&\\[8pt]
		\cline{3-11}
		
		&  & $I(G)$ & $A(\mathcal L(G))$ & $B(G)^T$ & \textbf{0 }& $2r-2$ & $2$ & $r$  &0&\\[8pt]
		\hline
		
		\multirow{3}{*}{4.} & \multirow{3}{*}{\parbox{1.7cm}{Central graph of $G$}}  & $V(G')$ & $A(\overline G)$ & $B(G)$ & \textbf{0 }& $n-r-1$ & $r$ &  2 &0&\\[8pt]
		\cline{3-11}
		
		& &$V(G)$ & $A(\overline G)$ & $B(G)$ & \textbf{0 } & $n-r-1$ & $r$ & 2  &0&\\[8pt]
		\cline{3-11}
		
		& & $I(G)$  & \textbf{0 } & $B(G)^T$ & $A(\overline G)$ & 0  &2 &  $r$  &$n-r-1$&\\[8pt]
		\hline
		
		\multirow{3}{*}{5.} & \multirow{3}{*}{\parbox{1.7cm}{Total graph of $G$}} & $V(G')$ & $A(G)$ & $B(G)$ & $A(\mathcal L(G))$ & $r$  & $r$ & $2$  &$2r-2$&\\[8pt]
		\cline{3-11}
		
		& &$V(G)$ & $A(G)$ & $B(G)$ & $A(\mathcal L(G))$ & $r$  & $r$ & $2$  &$2r-2$&\\[8pt]
		\cline{3-11}
		
		&  &$I(G)$ & $A(\mathcal L(G))$ & $B(G)^T$ & $A(G)$ & $2r-2$ & $2$ & $r$  & $r$&\\[8pt]
		\hline
		
		\multirow{3}{*}{6.} & \multirow{3}{*}{\parbox{1.7cm}{Quasi-total graph of $G$}} & $V(G')$ & $A(\overline G)$ & $B(G)$ & $A(\mathcal L(G))$ & $n-r-1$  & $r$ & $2$  &$2r-2$&\\[8pt]
		\cline{3-11}
		
		& &$V(G)$ & $A(\overline G)$ & $B(G)$ & $A(\mathcal L(G))$ & $n-r-1$  & $r$ & $2$  &$2r-2$&\\[8pt]
		\cline{3-11}
		
		&  &$I(G)$ & $A(\mathcal L(G))$ & $B(G)^T$ & $A(\overline G)$ & $2r-2$ & $2$ & $r$  & $n-r-1$&\\[8pt]
		\hline

		\multirow{3}{*}{7.} &\multirow{3}{*}{\parbox{1.7cm}{Duplicate graph of $G$ }} & $V(G')$ &\textbf{0 } & $A(G)$ & \textbf{0 } & $0$ & $r$ & $r$  &0&\\[8pt]
		\cline{3-11}
		
		& & $V(G)$ & \textbf{0 } & $A(G)$ & \textbf{0 } & $0$ & $r$ & $r$  &0&\\[8pt]
		\cline{3-11}
		
		& & $I(G)$  & \textbf{0 } & $A(G)$ & \textbf{0 } & 0  & $r$ & $r$  &$0$&\\[8pt]
		\hline
		
		\multirow{3}{*}{8.} & \multirow{3}{*}{\parbox{1.7cm}{$C$-graph of $G$}} & $V(G')$ & $A(G)$ & $I_n$ & \textbf{0 } & $r$ & $1$ & $1$  &0&\\[8pt]
		\cline{3-11}
		
		& & $V(G)$ & $A(G)$ & $I_n$ & \textbf{0 } & $r$ & $1$ & $1$  &0&\\[8pt]
		\cline{3-11}
		
		& & $I(G)$  & \textbf{0 } & $I_n$ & $A(G)$ & 0  & $1$ & 1  &$r$&\\[8pt]
		\hline
		
		\multirow{3}{*}{9.} &\multirow{3}{*}{\parbox{1.7cm}{$N$-graph of $G$ }} & $V(G')$ & $A(G)$ & $A(G)$ & \textbf{0 } & $r$ & $r$ & $r$  &0&\\[8pt]
		\cline{3-11}
		
		& & $V(G)$ & $A(G)$ & $A(G)$ & \textbf{0 } & $r$ & $r$ & $r$  &0&\\[8pt]
		\cline{3-11}
		
		& & $I(G)$  & \textbf{0 } & $A(G)$ & $A(G)$ & 0  & $r$ & $r$  &$r$&\\[8pt]
		\hline
		
		\multirow{3}{*}{10.} & \multirow{3}{*}{\parbox{1.7cm}{point complete subdivision graph of $G$}} & $V(G')$ & $J_n-I_n$ & $B(G)$ & \textbf{0} & $n-1$  & $r$ & $2$  &$0$&\\[8pt]
		\cline{3-11}
		
		& &$V(G)$ & $J_n-I_n$ & $B(G)$ & \textbf{0} & $n-1$  & $r$ & $2$  &$0$&\\[8pt]
		\cline{3-11}
		
		&  &$I(G)$ &  \textbf{0} & $B(G)^T$ & $J_n-I_n$ & $0$ & $2$ & $r$  & $n-1$&\\[8pt]
		\hline
		
		\multirow{3}{*}{11.} &  \multirow{3}{*}{\parbox{1.7cm}{$\mathcal Q$- complemented graph of $G$}}  & $V(G')$ & \textbf{0 } & $B(G)$ & $A(\overline{\mathcal L(G)})$ & 0  & $r$ & 2  & $m'-2r+1$ &\\[8pt]
		\cline{3-11}
		
		& & $V(G)$ &\textbf{0 } & $B(G)$ & $A(\overline{\mathcal L(G)})$ & 0  & $r$ & 2  &$m'-2r+1$&\\[8pt]
		\cline{3-11}
		
		&  & $I(G)$ & $A(\overline{\mathcal L(G)})$ & $B(G)^T$ & \textbf{0 }& $m'-2r+1$ & $2$ & $r$  &0&\\[8pt]
		\hline
		
		\multirow{3}{*}{12.} & \multirow{3}{*}{\parbox{1.7cm}{Total complemented graph of $G$}} & $V(G')$ & $A(G)$ & $B(G)$ & $A(\overline{\mathcal L(G)})$ & $r$  & $r$ & $2$  &$m'-2r+1$&\\[8pt]
		\cline{3-11}
		
		& &$V(G)$ & $A(G)$ & $B(G)$ & $A(\overline{\mathcal L(G)})$ & $r$  & $r$ & $2$  &$m'-2r+1$&\\[8pt]
		\cline{3-11}
		
		&  &$I(G)$ & $A(\overline{\mathcal L(G)})$ & $B(G)^T$ & $A(G)$ & $m'-2r+1$ & $2$ & $r$  & $r$&\\[8pt]
		\hline
		
		\multirow{3}{*}{13.} &  \multirow{3}{*}{\parbox{1.7cm}{Quasitotal complemented graph of $G$}}  & $V(G')$ & $A(\overline{G})$ & $B(G)$ & $A(\overline{\mathcal L(G)})$ & $n-r-1$  & $r$ & 2  &$m'-2r+1$&\\[8pt]
		\cline{3-11}
		
		& & $V(G)$ &$A(\overline{G})$ & $B(G)$ & $A(\overline{\mathcal L(G)})$ & $n-r-1$  & $r$ & 2  &$m'-2r+1$&\\[8pt]
		\cline{3-11}
		
		&  & $I(G)$ & $A(\overline{\mathcal L(G)})$ & $B(G)^T$ & $A(\overline{G})$& $m'-2r+1$ & $2$ & $r$  &$n-r-1$&\\[8pt]
		\hline 
		
		\multirow{3}{*}{14.} &  \multirow{3}{*}{\parbox{1.7cm}{Complete $\mathcal Q$- complemented graph of $G$}}  & $V(G')$ & $J_n-I_n$ & $B(G)$ & $A(\overline{\mathcal L(G)})$ & $n-1$  & $r$ & 2  &$m'-2r+1$&\\[8pt]
		\cline{3-11}
		
		& & $V(G)$ &$J_n-I_n$ & $B(G)$ & $A(\overline{\mathcal L(G)})$ & $n-1$  & $r$ & 2  &$m'-2r+1$&\\[8pt]
		\cline{3-11}
		
		&  & $I(G)$ & $A(\overline{\mathcal L(G)})$ & $B(G)^T$ & $J_n-I_n$& $m'-2r+1$ & $2$ & $r$  & $n-1$ &\\[8pt]
		\hline
		
		\multirow{3}{*}{15.} & \multirow{3}{*}{\parbox{1.7cm}{Complete subdivision graph of $G$}}  & $V(G')$
		&\textbf{0} & $B(G)$ & $J_m-I_m$ & 0 & $r$ & 2  & $m-1$ &\\[8pt]
		\cline{3-11}
		
		& & $V(G)$
		&\textbf{0} & $B(G)$ & $J_m-I_m$  & 0 & $r$ & 2  &  $m-1$ &\\[8pt]
		\cline{3-11}
		
		& & $I(G)$ &$J_m-I_m$& $B(G)^T$ &   \textbf{0 } & $m-1$ & 2 & $r$  &0&\\[8pt]
		\hline
		
		\multirow{3}{*}{16.} & 	\multirow{3}{*}{\parbox{1.7cm}{Complete $R$-graph of $G$}}  & $V(G')$ & $A(G)$ & $B(G)$ &  $J_m-I_m$ & $r$ & $r$ & 2 & $m-1$ & \\[8pt]
		\cline{3-11}
		
		&  & $V(G)$ & $A(G)$ & $B(G)$ & $J_m-I_m$ & $r$ & $r$ & 2 &$m-1$& \\[8pt]
		\cline{3-11}
		&   & $I(G)$ & $J_m-I_m$ & $B(G)^T$ & $A(G)$ & $m-1$ & 2 & $r$ &$r$ &\\[8pt]
		\hline
		
		\multirow{3}{*}{17.} & \multirow{3}{*}{\parbox{1.7cm}{Complete central graph of $G$}}  & $V(G')$ & $A(\overline G)$ & $B(G)$ & $J_m-I_m$ & $n-r-1$ & $r$ &  2 &$m-1$&\\[8pt]
		\cline{3-11}
		
		& &$V(G)$ & $A(\overline G)$ & $B(G)$ & \textbf{0 } & $n-r-1$ & $r$ & 2  &$m-1$&\\[8pt]
		\cline{3-11}
		
		& & $I(G)$  & $J_m-I_m$ & $B(G)^T$ & $A(\overline G)$ & $m-1$  &2 &  $r$  &$n-r-1$&\\[8pt]
		\hline
		
		\multirow{3}{*}{18.} & \multirow{3}{*}{\parbox{1.7cm}{Fully complete subdivision graph of $G$}}  & $V(G')$
		&$J_n-I_n$ & $B(G)$ & $J_m-I_m$ & $n-1$ & $r$ & 2  & $m-1$ &\\[8pt]
		\cline{3-11}
		
		& & $V(G)$
		&$J_n-I_n$ & $B(G)$ & $J_m-I_m$  & $n-1$ & $r$ & 2  &  $m-1$ &\\[8pt]
		\cline{3-11}
		
		& & $I(G)$ &$J_m-I_m$& $B(G)^T$ &   $J_n-I_n$ & $m-1$ & 2 & $r$  &$n-1$&\\[8pt]
		\hline
		
		\caption{The necessary entities to obtain the coronal of some graphs constrained by their vertex subsets}
		\label{table coronals of various graphs}
	\end{longtable}
\end{footnotesize}
\begin{cor}\label{Gammasemiregularbipartite}
	If $G$ is a semi-regular bipartite graph with bipartition $(X,Y)$ and parameters $(n_1,n_2,r_1,r_2)$, then we have the following.
	\begin{enumerate}[(1)]
		\item $\displaystyle\Gamma_G(x)=\frac{(n_1+n_2)x+2n_1r_1}{x^2-r_1r_2},$
		
		\item $\displaystyle\Gamma_G^X(x)=\frac{n_1x}{x^2-r_1r_2}.$
	\end{enumerate}
\end{cor}
\begin{proof}
	Notice that,
	\begin{eqnarray}
		A(G)=\begin{bmatrix}
			\textbf{0}_{n_1} & W_{n_1\times n_2} \\
			W_{n_2\times n_1} &  \textbf{0}_{n_2}
		\end{bmatrix}, \label{equation adjacency matrix of semiregular bipartite}
	\end{eqnarray}
	where $W\in \mathcal{RC}_{n_1\times n_2}(r_1,r_2)$.  Taking $a_1=0$, $a_2=r_1$, $a_3=r_2$ and $a_4=0$ in \eqref{equation coronalvalues of partitioned matrix} and \eqref{equation coronalvalues of partitioned matrix constrained by index two}, and using the fact that $n_1r_1=n_2r_2$, we get the proof of parts (1) and (2), respectively.	
\end{proof}

%
%

For two graphs $H_1$ and $H_2$, their \emph{join}, denoted by $H_1\vee H_2$, is the graph obtained by taking one copy of $H_1$ and $H_2$, and joining each vertex of $H_1$ to all the vertices of $H_2$.
\begin{cor}
	(\cite[Proposition~17]{mcleman2011}) If $H_1$ is an $r_1$-regular graph with $n_1$ vertices and $H_2$ is an $r_2$-regular graph with $n_2$ vertices, then
	\[\Gamma_{H_1\vee H_2}(x)=\displaystyle\frac{(n_1+n_2)x+n_1(n_2-r_2)+n_2(n_1-r_1)}{x^2-(r_1+r_2)x+r_1r_2-n_1n_2}.\]
\end{cor}

\begin{proof}
	Notice that \[A(H_1\vee H_2)=\begin{bmatrix}
	A(H_1) & J_{n_1\times n_2} \\ J_{n_2\times n_1} & A(H_2)
	\end{bmatrix}.\]
	Since $H_1, H_2$ are $r_1$, $r_2$-regular graphs, respectively, we have $A(H_1)\in \mathcal R_{n_1\times n_1}(r_1)$ and $A(H_2)\in \mathcal R_{n_2\times n_2}(r_2)$.  So, taking  $a_1=r_1, a_2=n_2, a_3=n_1$ and $a_4=r_2$ in \eqref{equation coronalvalues of partitioned matrix}, we obtain the result.
\end{proof}

\begin{cor}\label{GammaknT}
	If $T$ is a vertex subset of $K_n$ with $|T|=t$, then $\Gamma_{K_n}^{T}(x)=\displaystyle\frac{t(x-n+t+1)}{(x+1)(x-n+1)}$.
\end{cor}

\begin{proof}
	
	We arrange the rows and columns of $A(K_n)$ by the vertices in $T$ and the remaining vertices of $K_n$, respectively. Then we have
	\begin{equation}
		A(K_n)=\begin{bmatrix}
			A(K_t) &J_{t\times (n-t)}\\J_{(n-t)\times t} &A(K_{n-t})
		\end{bmatrix}.\label{adjacency matrix of Kn with partition}
	\end{equation}
	Taking $a_1=t-1$, $a_2=n-t$, $a_3=t$ and $a_4=n-t-1$ in \eqref{equation coronalvalues of partitioned matrix constrained by index two}, we get the result.
\end{proof}

The following result is established in \cite{schwenk1991}.
\begin{thm}\label{adjointmatrix of Kpq}
	(\cite[Theorem 4]{schwenk1991}) The adjoint matrix of $xI_{p+q}-A(K_{p,q})$ is given in the form of a partitioned matrix by
	
	$\begin{bmatrix}
	P_{K_{p-1,q}}(x)I_p+qx^{p+q-3}(J_p-I_p)&x^{p+q-2} J_{p\times q}\\x^{p+q-2} J_{q\times p}&P_{K_{p,q-1}}(x)I_q+px^{p+q-3}(J_q-I_q)
	\end{bmatrix}$.
\end{thm}

\begin{pro}\label{Gammakpq vertex subset}
	Consider the complete bipartite graph $K_{p,q}$ with a bipartition $(X,Y)$ be such that $|X|=p$. Let $S_1\subseteq X$ and $S_2\subseteq Y$ be such that $|S_1|=s_1$ and $|S_2|=s_2$.  Then
	$$\Gamma_{K_{p,q}}^{S_1\cup S_2}(x)=\displaystyle\frac{(s_1+s_2)x^2+2s_1s_2x-(s_1+s_2)pq+s_1^2q+s_2^2p}{x(x^2-pq)}.$$
\end{pro}

\begin{proof}
	Since, $\Gamma_{K_{p,q}}^{S_1\cup S_2}$ is the sum of all entries in the principal submatrix of $(xI_{p+q}-A(K_{p,q}))^{-1}$ formed by the vertices in $S_1\cup S_2$,
	by using Theorem \ref{adjointmatrix of Kpq}, we have
	\begin{align}
		\Gamma_{K_{p,q}}^{S_1\cup S_2}(x)=\displaystyle \frac{1}{P_{K_{p,q}}(x)}&\left( s_1 P_{K_{p-1,q}}(x) +q(s_1^2-s_1) x^{p+q-3}+2s_1s_2 x^{p+q-2}\right.\nonumber \\
		&\left.+s_2 P_{K_{p,q-1}}(x) +p(s_2^2-s_2) x^{p+q-3}\right) .
	\end{align}
	Using the fact that $P_{K_{p,q}}(x)=x^{p+q-2}(x^2-pq)$ in the above equation, we get the result.
\end{proof}

We deduce the following result, by taking $S_1=X$ and $S_2=Y$ in Proposition~\ref{Gammakpq vertex subset}.

\begin{cor}\label{Gammakpq}
	(\cite[Proposition~8]{mcleman2011}) $\Gamma_{K_{p,q}}(x)=\displaystyle\frac{(p+q)x+2pq}{x^2-pq}.$
\end{cor}

\section{The characteristic polynomial of the adjacency matrix of $G\circledast_{\mathcal {T}}\mathcal H$}

In the rest of the paper, we assume the following unless we specifically mention otherwise: $G$ is a graph with $V(G)=\{v_1,v_2,\ldots, v_n\}$, $\mathcal{H}=(H_1$ is a sequence of $n$ graphs $H_1,H_2,\ldots, H_n$ with $|V(H_i)|=h_i$ for $i=1,2,\ldots, n$ and $\mathcal T$ is a sequence of sets $T_1,T_2,\ldots,T_n$, where $T_i\subseteq V(H_i)$ with $|T_i|=t_i$ for $i=1,2,\ldots, n$. Let $\mathbf{r_i}:=\mathbf{r_{T_i}}$ for $i=1,2,\ldots,n$.

In this section, first we determine the characteristic polynomial of the adjacency  matrix of the generalized corona of $G$ and $\mathcal H$ constrained by $\mathcal T$, which is one of the main result of this paper.

The following result is used throughout this paper.
\begin{thm}\label{schur}(\cite{bapat2010})
	Let $A$ be a matrix partitioned as
	\[A=\begin{bmatrix}
	A_1&A_2\\A_3&A_4
	\end{bmatrix},
	\]where $A_1,A_4$ are square invertible matrices. Then
	\begin{eqnarray*}
		|A|&=|A_4| |A_1-A_2A_4^{-1}A_3|=|A_1| |A_4-A_3A_1^{-1}A_2|.
	\end{eqnarray*}
\end{thm}

\begin{thm}\label{chpoly gencorona}
	The characteristic polynomial of the adjacency matrix of $G\circledast_{\mathcal T}\mathcal{H}$ is
	
	\begin{eqnarray*}
		P_{G\circledast_{\mathcal T}\mathcal{H}}(x)=\left\{\prod_{i=1}^{n}P_{H_i}(x)\right\}
		\times \left| xI_n-A(G)-U_A\right|,
	\end{eqnarray*} where
	$ U_A=
	\begin{bmatrix}
	\Gamma_{H_1}^{T_1}(x) & 0& \cdots & 0
	\\
	0 & \Gamma_{H_2}^{T_2}(x) & \cdots & 0\\
	\vdots & \vdots & \ddots & \vdots\\
	0 & 0 &\cdots &\Gamma_{H_n}^{T_n}(x)
	\end{bmatrix}
	$.
\end{thm}

\begin{proof}
	We arrange the rows and columns of the adjacency matrix of  $G\circledast_{\mathcal {T}}\mathcal{H}$ by the vertices of $G,H_1,H_2,\ldots,H_n,$ respectively.
	Then
	\[A(G\circledast_{\mathcal T}\mathcal{H})=
	\begin{bmatrix}
	A(G) & C \\
	C^T & E
	\end{bmatrix},
	\]
	where
	\[
	E=
	\begin{bmatrix}
	A(H_1) & \textbf{0}& \cdots & \textbf{0} \\
	\textbf{0} & A(H_2) & \cdots & \textbf{0}\\
	\vdots & \vdots & \ddots & \vdots\\
	\textbf{0} & \textbf{0} &\cdots &A(H_n)
	\end{bmatrix}_{p\times p} ~~\text{and}~~C=
	\begin{bmatrix}
	\mathbf{r_1} & \textbf{0}& \cdots & \textbf{0} \\
	\textbf{	0} & \mathbf{r_2} & \cdots & \textbf{0}\\
	\vdots & \vdots & \ddots & \vdots\\
	\textbf{0}& \textbf{0 }&\cdots &\mathbf{r_n}
	\end{bmatrix}_{n\times p}.
	\]
	with $p=\displaystyle\sum_{i=1}^{n}h_i$.
	By using Theorem~\ref{schur}, we have
	\begin{eqnarray}
		P_{G\circledast_{\mathcal T}\mathcal{H}}(x)&=&
		\begin{vmatrix}
			xI_n-A(G) & -C \\
			-C^T & xI_p-E
		\end{vmatrix}\notag
		\\&=&\left|xI_p-E\right|\times\left|xI_n-A(G)-C(xI_p-E)^{-1}C^T\right|.\label{gendet}
	\end{eqnarray}
	
	It is not hard to see that,$$\left|xI_p-E\right|=\displaystyle\prod_{i=1}^{n}\left|xI_{h_i}-A(H_i)\right|=\displaystyle\prod_{i=1}^{n}P_{H_i}(x).$$
	
	Also, 	
	
	$C(xI_p-E)^{-1}C^T$
	\begin{eqnarray*}
		&=&C\begin{bmatrix}
			xI_{h_1}-A(H_1)  & \textbf{0}& \cdots & \textbf{0}
			\\
			\textbf{0} & xI_{h_2}-A(H_2) & \cdots & \textbf{0}\\
			\vdots & \vdots & \ddots & \vdots\\
			\textbf{0} & \textbf{0} &\cdots &xI_{h_n}-A(H_n)
		\end{bmatrix}^{-1}C^T
		\\&= &
		\begin{bmatrix}
			\mathbf{r_1}\left( xI_{h_1}-A(H_1)\right) ^{-1}\mathbf{r}^T_\mathbf{1} & 0& \cdots & 0
			\\
			0 & \mathbf{r_2}\left(xI_{h_2}-A(H_2)\right) ^{-1}\mathbf{r}^T_\mathbf{2} & \cdots & 0\\
			\vdots & \vdots & \ddots & \vdots\\
			0 & 0 &\cdots &\mathbf{r_n}\left(xI_{h_n}-A(H_n)\right) ^{-1}\mathbf{r}^T_\mathbf{n}
		\end{bmatrix}
		\\&= &U_A.
	\end{eqnarray*}
	\normalsize
	
	Substituting these values in~\eqref{gendet} we get the result.
\end{proof}

Theorem~\ref{chpoly gencorona} shows that, the $A$-spectrum of $G\circledast_{\mathcal T}\mathcal{H}$ can be completely determined by the $A$-spectrum of the constituent graphs and their coronals constrained by the corresponding vertex subsets. In the following result, we show that, if all the coronals of $H_i$'s constrained by their corresponding subsets $T_i$ are equal, then the $A$-spectrum of $G\circledast_{\mathcal T}\mathcal{H}$ is same regardless of the order of $H_i$'s in $\mathcal{H}$. So, in this case, by interchanging the order of $H_i$'s in $\mathcal H$,  we can get a family of $A$-cospectral graphs.
\begin{cor}\label{chpoly gen corona eqGamma}
	If  $\Gamma_{H_1}^{T_1}(x)=\Gamma_{H_2}^{T_2}(x)=\cdots=\Gamma_{H_n}^{T_n}(x)$, then the characteristic polynomial  of the adjacency matrix of $ G\circledast_{\mathcal T}\mathcal{H} $ is
	\[
	\left\{\prod_{i=1}^{n}P_{H_i}(x)\right\}\times	\left\{\prod_{j=1}^{n}\left( x-\lambda_j(G)-\Gamma_{H_1}^{T_1}(x)\right) \right\}
	.\]
\end{cor}

In the rest of this section, we consider some interesting graphs $H_i$'s whose adjacency matrices are $2\times2$ block matrices with some special constraints.
\begin{cor}\label{chpoly of gen corona vertex subsets}
	Suppose for $i=1,2,\ldots,n$,
	\[A(H_i)=\begin{bmatrix}
	A_{1i} & A_{2i} \\A_{2i}^T & A_{3i}
	\end{bmatrix},
	\]
	where $A_{1i}\in\mathcal R_{r_i\times r_i}(a_1)$, $A_{3i}\in\mathcal R_{s_i\times s_i}(a_4)$ and $A_{2i}\in \mathcal{RC}_{r_i\times s_i}(a_2,a_3)$, with $r_i+s_i=|V(H_i)|$.  Then we have the following.
	
	\begin{enumerate}[(1)]
		\item If $|V(H_1)|=|V(H_2)|=\cdots=|V(H_n)|=h$ and $r_1=r_2=\cdots=r_n=t$, then the characteristic polynomial of the adjacency matrix of $G\circledast\mathcal{H}$ is
		\begin{align}
			\displaystyle\frac{\displaystyle\prod_{i=1}^{n}P_{H_i}(x)}{\left\{ x^2-k_1x+k_2\right\}^n } \times& 	 \prod_{j=1}^{n}\left( x^3-\left\{k_1+\lambda_j(G)\right\}x^2+\left\{k_1\lambda_j(G)+k_2-h\right\}x\right.\nonumber\\&
			~~~~\left.-k_2\lambda_j(G)-t(a_2-a_4)-(h-t)(a_3-a_1)\right) ,
		\end{align}
		
		\item 	 If for each $i=1,2,\ldots,n$, $A_{1i}$ is the adjacency matrix of the subgraph induced by $T_i$ in $H_i$ and $|T_i|=t$, then the characteristic polynomial of the adjacency matrix of $G\circledast_{\mathcal T}\mathcal{H}$ is
		$$\displaystyle\frac{\displaystyle\prod_{i=1}^{n}P_{H_i}(x)}{\left\{ x^2-k_1x+k_2\right\}^n } \times 	 \prod_{j=1}^{n}\left( x^3-\left\{k_1+\lambda_j(G)\right\}x^2+\left\{k_1\lambda_j(G)+k_2-t\right\}x-k_2\lambda_j(G)+ta_4\right) ,$$
	\end{enumerate}
	where $k_1=a_1+a_4$, $k_2=a_1a_4-a_2a_3$.
\end{cor}

\begin{proof}
	Taking $n_1=t$ and $n_2=h-t$ in \eqref{equation coronalvalues of partitioned matrix} and \eqref{equation coronalvalues of partitioned matrix constrained by index two}, we have
	\begin{eqnarray}
		\displaystyle\Gamma_{H_i}(x)=\displaystyle\frac{hx+t(a_2-a_4)+(h-t)(a_3-a_1)}{x^2-(a_1+a_4)x+a_1a_4-a_2a_3}\notag
	\end{eqnarray}
	and
	\begin{eqnarray}
		\displaystyle\Gamma_{H_i}^{T_i}(x)=\frac{t(x-a_4)}{x^2-(a_1+a_4)x+a_1a_4-a_2a_3}\notag
	\end{eqnarray} for each $i=1,2,\ldots,n$.  So the proof follows bu using these values in Corollary \ref{chpoly gen corona eqGamma}.
\end{proof}

\begin{rem}\label{remark about adjacency corona vertes subdivision}
	\normalfont \begin{enumerate}[(1)]
		\item For each $i=1,2,\ldots,n$, let $H_i$ = $S(H_i')$, where $H_i'$ is an $r$-regular graph with $h$ vertices. Then we can obtain
		the characteristic polynomial of the adjacency matrix of $G\circledast\mathcal H$ by applying the values of $a_1, a_2,a_3,a_4$ as in the first row given against the subdivision graph in Table~\ref{table coronals of various graphs} and using the characteristic polynomial of the adjacency matrix of $S(H_i')$ \cite[(2.32)]{cvetkovic2010}, in Corollary~\ref{chpoly of gen corona vertex subsets};
		Further if $T_i=V(H_i')$ or $I(H_i')$ for each $i=1,2,\ldots,n$, then we can obtain the characteristic polynomial of the adjacency matrix of $G\circledast_{\mathcal T} \mathcal H$ by applying the values of $a_1, a_2,a_3,a_4$ as in second and third rows given against the subdivision graph in Table~\ref{table coronals of various graphs}, respectively, and using the characteristic polynomial of the adjacency matrix of $S(H_i')$, in Corollary~\ref{chpoly of gen corona vertex subsets}.
		
		\item  For each $i=1,2,\ldots,n$, if $H_i$ is one of the graph in $\mathcal U_{H_i'}$ for an $r$-regular graph $H_i'$ with $h$ vertices, and $T_i=V(H_i')$ or $I(H_i')$, then the characteristic polynomial of the adjacency matrix of $G\circledast\mathcal H$ and $G\circledast_{\mathcal T} \mathcal H$ can be obtained by the similar method described in the preceding part of this remark.
	\end{enumerate}
\end{rem}

\begin{cor}\label{chpoly gen corona semiregular}
	If $H_i$ is a semi-regular bipartite graph with bipartition $(X_i,Y_i)$ and parameters $(n_1,n_2,r_1,r_2)$ for $i=1,2,\ldots,n$, then we have the following.
	\begin{enumerate}[(1)]
		\item The characteristic polynomial of the adjacency matrix of $G\circledast\mathcal H$ is
		$$\displaystyle\frac{\displaystyle\prod_{i=1}^{n}P_{H_i}(x)}{\left\{ x^2-r_1r_2\right\}^n } 
		\times	 \prod_{j=1}^{n}\left( x^3-\lambda_j(G)x^2-\left\{r_1r_2+n_1+n_2\right\}x+r_1r_2\lambda_j(G)+2n_1r_1\right) .$$
		
		\item If $T_i=X_i$ for each $i=1,2,\ldots,n$, then the characteristic polynomial of the adjacency matrix of  $G\circledast_{\mathcal T}\mathcal H$ is
		$$\displaystyle\frac{\displaystyle\prod_{i=1}^{n}P_{H_i}(x)}{\left\{ x^2-r_1r_2\right\}^n }  \times 	 \prod_{j=1}^{n}\left( x^3-\lambda_j(G)x^2-\left\{n_1+r_1r_2\right\}x+r_1r_2\lambda_j(G)\right).$$
	\end{enumerate}
\end{cor}

\begin{proof}
	In view of \eqref{equation adjacency matrix of semiregular bipartite}, taking $t=n_1$, $h-t=n_2$, $a_1=0$, $a_2=r_1$, $a_3=r_2$ and $a_4=0$ in parts of (1) and (2) of Corollary~\ref{chpoly of gen corona vertex subsets}, we get the proof of parts (1) and (2), respectively.
\end{proof}

\begin{cor}\label{gencorona hcomplete vertex subset}
	If $H_i=K_m$ and $T_i$ is a vertex subset of $K_m$ with $|T_i|=t$ for each $i=1,2,\ldots,n$, then the $A$-spectrum of $G\circledast_{\mathcal T}\mathcal H$ is
	\begin{enumerate}[(i)]
		\item $-1$ with multiplicity $n(m-2)$;
		\item  for $i=1,2,\ldots,n$, the roots of the polynomial
		$$x^3-\left\{m+\lambda_i(G)-2\right\}x^2+\left\{(m-2)\lambda_i(G)-m+1-t\right\}x+(m-1)\lambda_i(G)+t(m-t-1).$$
	\end{enumerate}
\end{cor}
\begin{proof}
	In view of \eqref{adjacency matrix of Kn with partition}, taking $a_1=t-1$, $a_2=m-t$, $a_3=t$ and $a_4=m-t-1$ in Corollary~\ref{chpoly of gen corona vertex subsets}(2), we get the result.
\end{proof}

\begin{cor}\label{gencorona hcomplete bipartite vertex subset}
	Consider the complete bipartite graph $K_{p,q}$ with bipartition $(X,Y)$  such that $|X|=p$. Let $S_1\subseteq X$ and $S_2\subseteq Y$ with $|S_1|=s_1$ and $|S_2|=s_2$. If $H_i=K_{p,q}$ and $T_i=S_1\cup S_2$ for each $i=1,2,\ldots,n$, then the $A$-spectrum of $G\circledast_{\mathcal T}\mathcal H$ is
	\begin{enumerate}[(i)]
		\item $0$ with multiplicity $n(p+q-3)$,
		\item for $i=1,2,\ldots,n$, the roots of the polynomial
		$$x^4-\lambda_i(G)x^3-\left\{pq+s_1+s_2\right\}x^2+\left\{pq\lambda_i(G)-2s_1s_2\right\}x+(s_1+s_2)pq-s_1^2q-s_2^2p.$$
	\end{enumerate}
\end{cor}

\begin{proof}
	Applying Proposition~\ref{Gammakpq vertex subset} in Corollary~\ref{chpoly gen corona eqGamma}, we get the result.	 
\end{proof}

\section{The characteristic polynomial of the Laplacian matrix of $G\circledast_{\mathcal {T}}\mathcal H$}

\begin{notation}
	\normalfont
	Suppose $H$ is a graph with $h$ vertices, $T\subseteq V(H)$ and	$\mathbf{r}_T=(r_{1},r_{2},\ldots,r_{{h}})$, then we denote the diagonal matrix whose diagonal entries are $r_{1},r_{2},\ldots,r_{{h}}$ by $R_{T}$. Also the characteristic polynomial of $L(H)+R_T$ is denoted by $L_H^T(x)$.
\end{notation}

\begin{thm}\label{lchpoly gencorona}
	The characteristic polynomial of the Laplacian matrix of $G\circledast_{\mathcal T}\mathcal{H}$ is
	\begin{eqnarray*}
		L_{G\circledast_{\mathcal T}\mathcal{H}}(x)=\left\{\prod_{i=1}^{n}L_{H_i}^{T_i}(x)\right\}
		\times \left| xI_n-L(G)-U_L\right|,
	\end{eqnarray*} where
	$$ U_L=
	\begin{bmatrix}
	t_1+\Gamma_{L(H_1)+R_{T_1}}^{T_1}(x) & 0& \cdots & 0
	\\
	0 & t_2+\Gamma_{L(H_2)+R_{T_2}}^{T_2}(x) & \cdots & 0\\
	\vdots & \vdots & \ddots & \vdots\\
	0 & 0 &\cdots &t_n+\Gamma_{L(H_n)+R_{T_n}}^{T_n}(x)
	\end{bmatrix}.
	$$
\end{thm}

\begin{proof} Notice that
	\[L(G\circledast_{\mathcal T}\mathcal{H})=
	\begin{bmatrix}
	L(G)+N & -C \\
	-C^T & E'
	\end{bmatrix},
	\] where  $C$ is the matrix as in Theorem~\ref{chpoly gencorona}, $N=diag(t_1,t_2,\ldots, t_n)$  and
	\[
	E'= 	\begin{bmatrix}
	L(H_1)+R_{T_1} & \textbf{0}& \cdots & \textbf{0} \\
	\textbf{0} & L(H_2)+R_{T_2} & \cdots & \textbf{0}\\
	\vdots & \vdots & \ddots & \vdots\\
	\textbf{0} & \textbf{0} &\cdots &L(H_n)+R_{T_n}
	\end{bmatrix}_{p\times p}
	,\]
	with $p=\displaystyle\sum_{i=1}^{n}h_i$. By using Theorem~\ref{schur}, we have
	\begin{eqnarray}\label{mgendet}
		L_{G\circledast_{\mathcal T}\mathcal{H}}(x)&=&
		\begin{vmatrix}
			xI_n-L(G)- N & C \\
			C^T & xI_p- E'
		\end{vmatrix}\notag
		\\&=&\left|xI_p- E'\right|\times\left|xI_n-L(G)- N-C(xI_p-E')^{-1}C^T\right|.
	\end{eqnarray}
	
	It is not hard to see that,$$\left|xI_p- E'\right|=\displaystyle\prod_{i=1}^{n}\left|xI_{h_i}-L(H_i)-R_{T_i}\right|=\prod_{i=1}^{n}L_{H_i}^{T_i}(x).$$
	Also, \\
	$C(xI_p- E')^{-1}C^T$ 
	\begin{eqnarray*}
		& =&C\begin{bmatrix}
			xI_{h_1}-L(H_1)-R_{T_1}& \textbf{0}& \cdots &\textbf{ 0} &\\ \textbf{0} & xI_{h_2}-L(H_2)-R_{T_2} & \cdots & \textbf{0}\\
			\vdots & \vdots & \ddots & \vdots\\
			\textbf{	0} & \textbf{0} &\cdots &xI_{h_n}-L(H_n)-R_{T_n}
		\end{bmatrix}^{-1}C^T
		\\&=&	\begin{bmatrix}
			\Gamma_{L(H_1)+R_{T_1}}^{T_1}(x) & 0& \cdots & 0
			\\
			0 & \Gamma_{L(H_2)+R_{T_2}}^{T_2}(x) & \cdots & 0\\
			\vdots & \vdots & \ddots & \vdots\\
			0 & 0 &\cdots &\Gamma_{L(H_n)+R_{T_n}}^{T_n}(x)
		\end{bmatrix}.
	\end{eqnarray*}
	
	So we have, $N+C(xI_p- E')^{-1}C^T =U_L.$
	Substituting these values in~\eqref{mgendet} we get the result.
\end{proof}
\begin{note}
	\normalfont
	The characteristic polynomial of the signless Laplacian matrix of $G\circledast_{\mathcal T}\mathcal H$ can be obtained by using the analogous method described in Theorem~\ref{lchpoly gencorona}. Consequently, the rest of the results proved in this section can also be deduced for the signless Laplacian matrix (with additional constraints $A_{1i}\in \mathcal R_{t\times t}(a_1)$ and $A_{3i}\in \mathcal R_{(h_i-t)\times(h_i-t)}(a_4)$ in Corollary~\ref{lchpoly of gen corona vertex subsets}). The details are omitted.
\end{note}
Theorem~\ref{lchpoly gencorona} shows that, the characteristic polynomial of the Laplacian matrix of $G\circledast_{\mathcal T}\mathcal{H}$ can be completely determined by the $L$-spectrum of $G$, the polynomials $L_{H_i}^{T_i}(x)$ and the coronals of the matrices $L(H_i)+R_{T_i}$ constrained by their vertex subsets. The following is a direct consequence of Theorem~\ref{lchpoly gencorona}, which shows that, if all the coronals $\Gamma_{L(H_i)+R_{T_i}}^{T_i}(x)$ are equal, then the $L$-spectrum of $G\circledast_{\mathcal T}\mathcal{H}$ is same regardless of the order of $H_i$'s in $\mathcal{H}$. So, in this case, by interchanging the order of $H_i$'s in $\mathcal H$,  we can get a family of $L$-cospectral graphs.
\begin{cor}\label{lchpoly gen corona eqGamma}
	If  $|T_1|=|T_2|=\cdots=|T_n|=t$ and $\Gamma_{L(H_1)+R_{T_1}}^{T_1}(x)=\Gamma_{L(H_2)+R_{T_2}}^{T_2}(x)=\cdots=\Gamma_{L(H_n)+R_{T_n}}^{T_n}(x)$, then the characteristic polynomial of the Laplacian matrix of $G\circledast_{\mathcal T}\mathcal{H}$ is
	$$
	\left\{\prod_{i=1}^{n}L_{H_i}^{T_i}(x)\right\}
	\times \left\{\prod_{j=1}^{n}\left( x-t-\mu_j(G)-\Gamma_{L(H_1)+R_{T_1}}^{T_1}(x)\right) \right\}.$$
\end{cor}


\begin{cor}\label{lchpoly of gen corona vertex subsets}
	Let $|T_1|=|T_2|=\cdots=|T_n|=t$ and
	\[A(H_i)=\begin{bmatrix}
	A_{1i} & A_{2i} \\A_{2i}^T & A_{3i}
	\end{bmatrix},
	\]
	where $A_{1i}$ is the adjacency matrix of the subgraph induced by $T_i$ in $H_i$ and $A_{2i}\in \mathcal{RC}_{t\times (h_i-t)}(a_2,a_3)$ for $i=1,2,\ldots,n$.
	Then the characteristic polynomial of the Laplacian matrix of  $G\circledast_{\mathcal T}\mathcal{H}$ is
	$$\displaystyle
	\left\{\frac{1}{ x^2-sx+a_3 }\right\}^n \times \left\{\prod_{i=1}^{n}L_{H_i}^{T_i}(x)\right\}$$ $$\times \prod_{j=1}^{n}\left(  x^3-\left\{s+t+\mu_j(G)\right\}x^2+\left\{s\left( t+\mu_j(G)\right) +a_3-t\right\}x-a_3\mu_j(G)\right) ,$$
	where $s=a_2+a_3+1$.
\end{cor}

\begin{proof}
	Let $H_i'$ be the subgraph induced by $T_i$ and $H_i''$ be the subgraph induced by $V(H_i)\setminus T_i$ of $H_i$
	for $i=1,2,\ldots,n$. Then we have,
	\[L(H_i)=
	\begin{bmatrix}
	L(H_i')+a_2I_{t} & -A_{2i} \\-A_{2i}^T & L(H_i'')+a_3I_{h_i-t}
	\end{bmatrix}.\]
	Also notice that $R_{T_i}=\begin{bmatrix}
	I_{t} & 0\\0 &0
	\end{bmatrix}$.
	So,  \[L(H_i)+R_{T_i}=
	\begin{bmatrix}
	L(H_i')+(a_2+1)I_{t} & -A_{2i} \\-A_{2i}^T & 	L(H_i'')+a_3I_{h_i-t}
	\end{bmatrix}.\]
	
	\noindent	Taking  $a_2+1$, $-a_2$, $-a_3$, $a_3$  and $t$ in place of $a_1$, $a_2$, $a_3$, $a_4$ and $n_1$, respectively in Theorem~\ref{equation coronalvalues of partitioned matrix constrained by index two}, we have $$\displaystyle\Gamma_{L(H_i)+R_{T_i}}^{T_i}(x)=\frac{t(x-a_3)}{x^2-(a_2+a_3+1)x+a_3}$$ for $i=1,2,\ldots,n$. Applying this value in Corollary~\ref{lchpoly gen corona eqGamma}, we obtain the result.	
\end{proof}

In the following result, we determine $L_{H}^T(x)$ for the graphs obtained by some unary operations and some subsets $T$.
\begin{pro}\label{polynomial LHT for partitioned with 0}
	Let $H$ be a graph with $n$ vertices, $T\subseteq V(H)$ with $|T|=t$, and
	\[A(H)=\begin{bmatrix}
	A_{1} & A_{2} \\A_{2}^T & A_{4}
	\end{bmatrix},
	\]
	where $A_{1}$ and $A_4$ are the adjacency matrices of the subgraphs $F_1$ and $F_2$ of $H$ induced by $T$ and $V(H)\setminus T$, respectively and $A_{2}\in \mathcal{RC}_{t\times (n-t)}(a_2,a_3)$, where $a_3\neq 0$. If $A_4=t_1I_{n-t}+t_2J_{n-t}+t_3A_2^TA_2$, 
	then
	{\footnotesize \begin{align}
			L_H^T(x)=(x-c)^{n-2t} \left[x^2-\left((n-t)t_2+a_3-\frac{t_2}{a_3}ta_2+a_2+1\right)x+(a_2+1)\left([n-t]t_2+a_3-\frac{t_2}{a_3}ta_2\right) \right]\nonumber\\
			\times
			\prod_{i=2}^{t}\left[x^2-\left(c-t_3\lambda_i(A_2A_2^T)+a_2+\mu_i(F_1)+1\right)x+(c+t_3\lambda_i(A_2A_2^T))(a_2+\mu_i(F_1)+1)-\lambda_i(A_{2}A_{2}^T) \right]
	\end{align}}
	where  $\mu_i(F_1)$, $\lambda_i(A_{2}A_{2}^T)$ are eigenvalues   corresponding to a common eigenvector of $L(F_1)$ and $A_{2}A_{2}^T$, respectively for each $i=1,2,\ldots , n$ and $c=t_2(n-t)+t_3a_2a_3+a_3$.
\end{pro}

\begin{proof}
	It can be verified that  \[L(H)+R_{T}=
	\begin{bmatrix}
	L(F_1)+(a_2+1)I_{t} & -A_{2} \\-A_{2}^T & cI_{n-t}-t_2J_{n-t}-t_3A_2^TA_2
	\end{bmatrix}.\]
	Then
	
	$L_H^T(x)$
	\begin{eqnarray}	
		&=&
		\begin{vmatrix}
			(x-a_2-1)I_t-L(F_1) & A_{2} \\A_{2}^T & (x-c)I_{n-t}+t_2J_{n-t}+t_3A_2^TA_2
		\end{vmatrix}\notag
		\\&=&	\begin{vmatrix}
			(x-a_2-1)I_t-L(F_1) & A_2\\A_2^T-\displaystyle t_3A_2^T\left((x-a_2-1)I_t-L(F_1)\right)& (x-c)I_ {n-t}+t_2J_{n-t}
		\end{vmatrix} \notag \\&&~~~~~~~~~~~~~~~~~~~~~~~~~~~~~~~~~~~~~~~~~~~~~~~~~~~~~~~~~~~~~~~~~~~~~R_2\rightarrow R_2-t_3A_2^TR_1 \notag 
		\\ &=&
		\begin{vmatrix}
			(x-a_2-1)I_t-L(F_1) & A_2\\A_2^T-\left\{\displaystyle t_3A_2^T+\frac{t_2}{a_3}J_{m\times n}\right\}\left\{	(x-a_2-1)I_t-L(F_1)\right\}& (x-c)I_{n-t}
		\end{vmatrix}\notag \\ &&~~~~~~~~~~~~~~~~~~~~~~~~~~~~~~~~~~~~~~~~~~~~~~~~~~~~~~~~~~~~~~~~~~~~~R_2\rightarrow R_2-\frac{t_2}{a_3}J_{t\times (n-t)}R_1\notag\\
		&=&(x-c)^{n-2t}\notag\\&&
		\times\left|\left\{(x-c)I_t-t_3A_2A_2^T-\frac{t_2}{a_3}a_2J_t\right\}\left\{[x-(a_2+1)]I_t-L(F_{1})\right\}-A_{2}A_{2}^T\right|\label{equation1}.
	\end{eqnarray}	
	Since $L(F_1)$ and $A_2A_2^T $ commutes with each other, so by \cite[Proposition 2.3.2]{hei2011}, there exists orthonormal vectors $x_1,x_2,\ldots,x_n$ such that $x_i$'s are eigenvectors of both $L(F_1)$ and $A_2A_2^T$.
	Let $P$ be the matrix whose columns are $x_1,x_2,\ldots,x_n$. Then we have
	\begin{eqnarray}
		P^TL(F_1)P=diag(\mu_1(F_1),\mu_2(F_1),\dots,\mu_n(F_1))\label{equation2}
	\end{eqnarray}
	and
	\begin{eqnarray}
		P^T(A_2A_2^T)P=diag(\lambda_1(A_2A_2^T),\lambda_2(A_2A_2^T),\dots,\lambda_n(A_2A_2^T))\label{equation3}.
	\end{eqnarray}
	So, \eqref{equation1} becomes
	
	\noindent	$L_H^T(x)$
	{\small \begin{align}	
			=&(x-c)^{n-2t}\notag \\&\times\left| P^T\right| \left|\left\{(x-c)I_t-t_3A_2A_2^T-\frac{t_2}{a_3}a_2J_t\right\}\left\{[x-(a_2+1)]I_t-L(F_{1})\right\}-A_{2}A_{2}^T\right||P|\nonumber
			\\=&
			(x-c)^{n-2t}\nonumber
			\\&\times \left|\left\{(x-c)I_t-t_3A_2A_2^T-\frac{t_2}{a_3}a_2P^TJ_tP\right\}\left\{[x-(a_2+1)]I_t-P^TL(F_{1})P\right\}-P^TA_{2}A_{2}^TP\right|\notag\\\label{equation4}
	\end{align}}
	Using \eqref{equation2} and \eqref{equation3} in \eqref{equation4}, we obtain the result.
\end{proof}

\begin{rem}\label{remark lchpoly}
	\normalfont
	\begin{enumerate}[(1)]
		\item If $H'$ is an $r$-regular graph with $h$ vertices, and $H=S(H')$, then the polynomial $L_H^T(x)$, where  $T=V(H')$ (resp. $I(H')$) can be obtained as follows: Taking the matrices $A_1, A_2, A_3, A_4$ and the values $a_2,a_3$ as mentioned in second row (resp. third row) given against the subdivision graph in Table~\ref{table coronals of various graphs}, and substitute these values in Proposition~\ref{polynomial LHT for partitioned with 0}(1) (resp. Proposition~\ref{polynomial LHT for partitioned with 0}(2)).
		
		\item If for each $i=1,2,\ldots,n$,  $H_i'$ is an $r$-regular graph with $h$ vertices, $H_i=S(H_i')$ and $T_i=V(H_i)$ (resp. $I(H')$)  then we can obtain the characteristic polynomial of the Laplacian matrix of $G\circledast_{\mathcal T}\mathcal H$ as follows: First find the polynomials $L_{H_i}^{T_i}(x)$ for each $i=1,2,\ldots,n$ as mentioned in the preceding part of this remark. Apply these polynomials and the values $a_2,a_3$ as mentioned in the second row (resp. third row) given against the subdivision graph as in Table~\ref{table coronals of various graphs}, in Corollary~\ref{lchpoly of gen corona vertex subsets}.
		
		\item If for each $i=1,2,\ldots,n$,  $H_i'$ is an $r$-regular graph with $h$ vertices, $H_i$ is one of the graph in $\mathcal U_{H_i'}$, and $T_i=V(H_i)$ (resp. $I(H')$),  then we can obtain the characteristic polynomial of the Laplacian matrix of $G\circledast_{\mathcal T}\mathcal H$ by using a similar method as described in the preceding part of this remark.
	\end{enumerate}
\end{rem}

\begin{cor}\label{lchpoly gencorona hcomplete vertex subset}
	If $H_i=K_m$ and $T_i$ is a vertex subset of $K_m$ with $|T_i|=t$ for $i=1,2,\ldots,n$, then the characteristic polynomial of the Laplacian matrix of $G\circledast_{\mathcal T}\mathcal H$ is
	\begin{align}
		&~~~~~~~~~~~~~~~~~~~~~~~~~(x-m)^{n(m-t-1)}(x-m-1)^{n(t-1)}
		\nonumber\\&\times \displaystyle \prod_{i=1}^{n}\left( x^3-\left\{t+\mu_i(G)+m+1\right\}x^2-(m+1)\left( t+\mu_i(G)\right) x-t\mu_i(G)\right) 
	\end{align}
\end{cor}
\begin{proof}
	In view of \eqref{adjacency matrix of Kn with partition}, taking $a_1=t-1$, $a_2=m-t$, $a_3=t$ and $a_4=m-t-1$ and by using the Laplacian spectrum of $L(K_t)$ in Proposition~\ref{polynomial LHT for partitioned with 0}, we have
	
	\begin{eqnarray}
		L_{K_m}^{T_i}(x)&=&
		(x-m)^{m-t-1} (x-m-1)^{t-1}(x^2-(m+1)x+t). \notag
	\end{eqnarray}
	Using the above identity, in Corollary \ref{lchpoly of gen corona vertex subsets}, we obtain the result.
\end{proof}

\begin{cor}
	Let $H_i$ be a semi-regular bipartite graphs with bipartition $(X_i,Y_i)$, parameters $(n_1,n_2,r_1,r_2)$. If \[A(H_i)=\begin{bmatrix}
	\textbf{0}_{n_1} & W_{n_1\times n_2} \\
	W_{n_2\times n_1} &  \textbf{0}_{n_2}
	\end{bmatrix},\] and $T_i=X_i$ for each $i=1,2,\ldots,n$, then the characteristic polynomial of the Laplacian matrix of $G\circledast_{\mathcal T}\mathcal H$ is
	\begin{align}
		(x-r_2)^{n(n_2-n_1)}\times\displaystyle
		\left\{\prod_{i=1}^{n}\prod_{j=2}^{n_1}\left(x^2-sx+r_2(r_1+1)-\lambda_j(W_iW_i^T) \right)\right\} \nonumber \\\times \left\{\displaystyle
		\prod_{i=1}^{n}\left(x^3-[s+b_i]x^2+\left\{sb_i+r_2-n_1\right\}x-r_2\mu_i(G) \right)\right\}
	\end{align},
	where $s=r_1+r_2+1$ and $b_i=n_1+\mu_i(G)$.
\end{cor}

\begin{proof}
	Notice that $W_i\in \mathcal{RC}_{n_1\times n_2}(r_1,r_2)$ for $i=1,2,\ldots,n$. So
	taking  $A_1=0=A_4$, $A_2=W_i$, $a_2=r_1$ and $a_3=r_2$ in Proposition~\ref{polynomial LHT for partitioned with 0}, we get
	\begin{eqnarray}
		L_{H_i}^{T_i}(x)
		&=&(x-r_2)^{n_2-n_1}\times\prod_{j=1}^{n_1}\left(x^2-sx+r_2(r_1+1)-\lambda_j(W_iW_i^T) \right).\notag
	\end{eqnarray}
	By taking $t=n_1$, $a_2=r_1$ and $a_3=r_2$ in Corollary~\ref{lchpoly of gen corona vertex subsets} and using the above identity and the fact $\lambda_1(W_iW_i^T)=r_1r_2$, we obtain the result.
\end{proof}

Since $K_{p,q}$ is a semi-regular bipartite graph with parameter $(p,q,q,p)$, the following is a direct consequence  of the preceding result.
\begin{cor}
	Consider the complete bipartite graph $K_{p,q}$ with bipartition $(X,Y)$ such that $|X|=p$. If $H_i\cong K_{p,q}$ and $T_i=X$ for $i=1,2,\ldots,n$, then the $L$-spectrum of $G\circledast_{\mathcal T}\mathcal H$ is
	\begin{enumerate}[(i)]
		\item $p+1$ with multiplicity $p-1$;
		\item $q$ with multiplicity $q-1$;
		\item for $i=1,2,\ldots,n$, the roots of the polynomials
		$x^3-[2p+\mu_i(G)+q+1]x^2+\{(p+q+1)(p+\mu_i(G))\}x-q\mu_i(G) $.
	\end{enumerate}
\end{cor}

\begin{rem}
	\normalfont	
	As particular cases of the results we proved so far in this section and in the previous section, we can deduce the characteristic polynomials of the adjacency and the Laplacian matrices of some variants of corona of graphs defined in the literature: We can deduce \cite[Theorems~3.1  and 4.1]{laali2016}, in which the characteristic polynomials of the adjacency and the Laplacian matrices of the generalized corona of $G$ and $\mathcal H$ are described, by taking $T_j=V(H_j)$ for $j=1,2,\ldots,n$ in Theorem~\ref{chpoly gencorona} and Theorem~\ref{lchpoly gencorona}. Consequently, we can deduce \cite[Theorem~2]{mcleman2011} in which the characteristic polynomial of the adjacency matrix of the corona of $G$ and $H$ is obtained \cite[Theorems~3.1 and 3.2]{barik2007} in which the $A$-spectrum (when $H$ is regular) and the $L$-spectrum of $G$ and $H$ are determined; Also the characteristic polynomials of the adjacency and the Laplacian matrices of the corona-vertex subdivision graph of $G$ and $H$, and the corona-edge subdivision graph of $G$ and $H$ \cite{lu2014} can be deduced by taking $H_i\cong H$ in Remarks~\ref{remark about adjacency corona vertes subdivision}(1) and \ref{remark lchpoly}(2).
\end{rem}

In the following result, we obtain the characteristic polynomials of the adjacency and the Laplacian matrices of cluster of two graphs $G$ and $H$, by taking $H_i\cong H$ and $T_i=\{u\},$ $i=1,2,\ldots,n$ in Corollaries~\ref{chpoly gen corona eqGamma} and \ref{lchpoly gen corona eqGamma}, respectively.

\begin{cor}\label{chpolycluster}
	Let $G$ be a graph with $n$ vertices and $H$ be a rooted graph with root vertex $u$. Then we have the following:
	\begin{enumerate}[(1)]
		\item  The characteristic polynomial of the adjacency matrix of ${G\{H\}}$ is
		$$\left\{P_{H}(x)\right\}^n P_G(x-\Gamma_{H}^u(x)).$$
		
		\item The characteristic polynomial of the Laplacian  matrix of ${G\{H\}}$ is
		$$
		\left\{L_H^u(x)\right\}^n
		\times
		L_G\left(x-1- \Gamma_{L(H)+R}^{u}(x)\right)
		,$$
		where $R$ is the matrix whose diagonal entry corresponding to the vertex $u$ is $1$ and all other entries are $0$.
	\end{enumerate}
\end{cor}

\section{Conclusions}
In this paper, we introduced a new generalization of corona of graphs in which the
base graphs are joined to the vertices in a vertex subset of the constituent graphs
instead of joining all the vertices. 
Further, it generalizes some existing corona operations defined in the literature.

Also, we defined some more variants of corona operations. Further, we introduced the notion of the coronal of a matrix constrained by an index set. By using this, we determined the characteristic polynomials of the adjacency and the Laplacian matrices of the generalized corona of graphs constrained by vertex subsets. The significance of these results is that they provide a simple and effective way to deduce the characteristic polynomials of the adjacency and the Laplacian matrices of the above mentioned existing corona of graphs as well as new variants of corona of graphs. 

We have introduced the notion of coronal of a matrix constrained by an
index set and the coronal of a graph constrained by vertex subsets. This value enables
us to determine the characteristic polynomials of the adjacency, the Laplacian and the
signless Laplacian matrices of the graphs constructed by the M-generalized corona of
graphs constrained by the vertex subsets. We determine the coronal of a matrix having
some specific properties constrained by some index sets. By using that results, we
have determined the coronals of the graphs constructed by the unary graph operations
defined in this thesis and some well-known graphs

We can obtain the number of spanning trees and the Kirchhoff index of the new variants of corona of graphs by using Remark~\ref{remark lchpoly}.

The determination of the characteristic polynomials of the other graph matrices such as normalized Lapalacian and distance matrices of the graph obtained by the generalized corona of graphs constrained by vertex subsets are further research problems.

\section*{Acknowledgment}
The second author is supported by INSPIRE Fellowship, DST, Government of India under the grant no. DST/INSPIRE Fellowship/[IF150651] 2015.

\end{document}